\documentclass[11pt,reqno]{amsart}

\usepackage[usenames,dvipsnames]{xcolor}
\RequirePackage[OT1]{fontenc}

\usepackage[colorlinks,citecolor=blue,urlcolor=blue,linkcolor=blue,linktocpage=true]{hyperref}

\usepackage[]{amsmath}
\usepackage{amssymb,amsthm,amsfonts,amsbsy,latexsym,amsxtra}

\usepackage{graphicx}

\usepackage{todonotes}
\usepackage{appendix}
\usepackage{url}

\usepackage[usenames,dvipsnames]{xcolor}            
\usepackage{graphicx}
\usepackage{bbm,bm}
\usepackage{comment}
\usepackage{mathtools}
\usepackage{placeins}
\usepackage[labelfont={sl}]{caption}
\usepackage[labelfont={normalfont}]{subcaption}

\usepackage{tikz}
\usetikzlibrary{decorations.pathreplacing}
\usetikzlibrary{calc}

\usepackage{amssymb,amsfonts,amsmath,amsthm,amscd,dsfont,mathrsfs}

\usepackage{url}

\usepackage{enumerate}
\usepackage{hyperref}

\numberwithin{equation}{section}

\newtheorem{theorem}{Theorem}[section]

\newtheorem{prop}[theorem]{Proposition}
\newtheorem{lemma}[theorem]{Lemma}
\newtheorem{corollary}[theorem]{Corollary}

\theoremstyle{definition}

\newtheorem{remark}[theorem]{Remark}

\newcommand{\E}{\mathbf{E}}
\newcommand{\R}{\mathbb{R}}

\newcommand{\N}{\mathbb{N}}
\newcommand{\A}{\mathcal{A}}

\newcommand{\eps}{\varepsilon}

\newcommand{\Prob}[1]{\mathbf P\{#1\}}

\usepackage{color}

\newcommand{\qmq}[1]{\quad \mbox{#1} \quad}
\newcommand{\qm}[1]{\quad \mbox{#1}}
\newcommand{\bdot}{\boldsymbol{\cdot}}

\allowdisplaybreaks[4]

\begin{document}

\title[Dickman Approximation in the Kolmogorov distance]{Dickman Approximation of weighted random sums in the Kolmogorov distance}

\author[C.\ Bhattacharjee]{Chinmoy Bhattacharjee}
\address{Department of Mathematics, University Luxembourg, Maison du Nombre, 6 avenue de la Fonte, 4364 Esch-sue-Alzette, Luxembourg}
\email{chinmoy.bhattacharjee@uni.lu}
\author[M.\ Schulte]{Matthias Schulte}
\address{Hamburg University of Technology, Institute of Mathematics, Am Schwarzenberg-Campus 3, 
	21073 Hamburg, Germany}
\email{matthias.schulte@tuhh.de}


\begin{abstract}
	We consider distributional approximation by generalized Dickman distributions, which appear in number theory, perpetuities, logarithmic combinatorial structures and many other areas. We prove bounds in the Kolmogorov distance for the approximation of certain weighted sums of Bernoulli and Poisson random variables by members of this family. While such results have previously been shown in Bhattacharjee and Goldstein (2019) for distances based on smoother test functions and for a special case of the random variables considered in this paper, results in the Kolmogorov distance are new. We also establish optimality of our rates of convergence by deriving lower bounds. As a result, some interesting phase transitions emerge depending on the choice of the underlying parameters. The proofs of our results mainly rely on the use of Stein's method. In particular, we study the solutions of the Stein equation corresponding to the test functions associated to the Kolmogorov distance, and establish their smoothness properties. As applications, we study the runtime of the Quickselect algorithm and the weighted depth in randomly grown simple increasing trees.
\end{abstract}

\maketitle

\section{Introduction and main results}
The Dickman distribution with parameter $\theta>0$ can be described as the unique non-negative fixed point of the distributional transformation $W \mapsto W^*$ given by
\begin{equation}\label{def:Dickman}
W^*=_d U^{1/\theta}(W+1)
\end{equation}
where $U \sim\mathbb{U}[0,1]$ is a uniform random variable on the interval $[0,1]$  independent of $W$, with $=_d$ denoting equality in distribution. The case when $\theta=1$ corresponds to the standard Dickman distribution, which first appeared in the work of Dickman \cite{Di30} in the context of \textit{smooth numbers}. Since then, the Dickman distributions have arisen in many areas like certain count statistics in logarithmic and quasi-logarithmic combinatorial structures \cite{ABT_article,ABT,BaNi11}, power weighted sums of lengths of edges joining minimal points to the origin in minimal directed spanning trees \cite{PW04}, certain sums of independent random variables \cite{Pi16,Pi16b}, and the runtime of Quickselect algorithm \cite{Hwa,Go17}, to name a few. In the context of Vervaat perpetuities, see \cite{Ve79}, one interprets \eqref{def:Dickman} as the relation between two consecutive values of a perpetuity, so that the Dickman distributions are their stationary distributions.

Given its appearance in such varied applications, not surprisingly, there has been considerable interest in proving bounds for Dickman approximations \cite{Go17v1,Go17,AMPS,BG17}. To obtain such a result, one needs a notion of distance between probability distributions. For two non-negative random variables $X$ and $Y$ and a large class of measurable real-valued test functions $\mathcal{H}$ on $\R_+:=[0,\infty)$, define the distance $d_{\mathcal{H}}$ between (the distributions of) $X$ and $Y$ as
\begin{equation}\label{eq:gendis}
	d_{\mathcal{H}}(X,Y) =\sup_{h \in \mathcal{H}} |\E\, [h(X)]-\E\, [h(Y)]|.
\end{equation}
Since we are only interested in the comparison of non-negative random variables throughout this paper, we introduce all distances in terms of test functions on $\mathbb{R}_+$, but by extending their domain to $\mathbb{R}$, all the distances can be defined for general real-valued random variables. The so-called Wasserstein distance is defined by setting $\mathcal{H}={\rm Lip}_1$ in \eqref{eq:gendis}, where for $\alpha \ge 0$,
\begin{equation*}
	{\rm Lip}_\alpha = \{h:\mathbb{R}_+\to\mathbb{R} \text{ with } |h(x)-h(y)| \le \alpha |x-y| \,\, \forall \,\, x,y\in\mathbb{R}_+\}.
\end{equation*} 
In \cite{Go17v1}, Stein's method was used to develop a machinery to prove bounds for the approximation by Dickman distributions in the Wasserstein distance. The argument in \cite{Go17v1} needed one to construct a coupling of a random variable with a distributional transform of it, the so-called \textit{Dickman bias transform}. While this could be achieved for the runtime of Quickselect algorithm, for other examples of Dickman convergences, this coupling appeared elusive. In \cite{BG17}, an alternative approach using a different Stein equation was proposed which does not require the construction of such a coupling and allows one to handle several other examples. But the method in \cite{BG17} leads only to bounds in a distance weaker than the Wasserstein distance, the so-called Wasserstein-2 distance $d_{1,1}$, which is obtained by taking the class of test functions 
\begin{equation}\label{eqn:TestFunctionsH11}
\mathcal{H}=\mathcal{H}_{1,1} := \{h : \R_+ \to \R \text{ with } h \in {\rm Lip}_1, h' \in {\rm Lip}_1\}
\end{equation}
in \eqref{eq:gendis}. For the even weaker Wasserstein-3 distance, where the test functions are required to have three bounded derivatives, Dickman approximation via the Stein-Tikhomirov method was considered in \cite{AMPS}.

In the present work, our aim is to develop a framework to provide bounds in the Kolmogorov distance for the Dickman approximation. Denote by $\mathcal{H}_K$ the class of functions 
\begin{equation}\label{eq:dK}
\mathcal{H}_K=\{\mathbb{R}_+\ni x\mapsto \mathds{1}(x \le a): a\in \R_+\}.
\end{equation}
The Kolmogorov distance between (the distributions of) two non-negative random variables $X$ and $Y$ is defined as
\begin{equation*}
d_K(X,Y)=\sup_{h \in \mathcal{H}_K} |\E\,[h(X)]-\E \,[h(Y)]|=\sup_{a \in \R_+} |\Prob{X \le a}-\Prob{Y \le a}|,
\end{equation*}
which is the supremum norm of the difference of the cumulative distribution functions of $X$ and $Y$. Thus, the Kolmogorov distance is much more straightforward to interpret than the Wasserstein or the Wasserstein-2 distance.  It is arguably the most prominent distance between probability distributions and is also considered in the classical Berry-Esseen inequality for the normal approximation. Due to the lack of smoothness of the test functions in $\mathcal{H}_K$, it is often very challenging to obtain bounds in the Kolmogorov distance. One of the main achievements of our paper is to provide a general method to prove bounds on the Kolmogorov distance for Dickman convergences.

For independent $B_k \sim {\rm Ber}(1/k)$, $k \in \N$, i.e., $\Prob{B_k=1}=1/k=1-\Prob{B_k=0}$, the random variable $\sum_{k=1}^n k B_k$, and more general weighted Bernoulli sums have appeared in many applications, e.g.\ in the context of the sum of positions of \textit{records} in a uniformly random permutation of $\{1,\dots,n\}$ (see \cite{Re62}), runtime of Quickselect algorithm \cite{Hwa}, and weighted distances and weighted depths in randomly grown simple increasing trees \cite{KP07}. It is well-known that the random variable $n^{-1}\sum_{k=1}^n k B_k$ converges in distribution to the standard Dickman distribution (see e.g.\ \cite[Theorem~6.3]{ABT}). 
For this convergence, recently in \cite{BG17}, a bound of the order $1/n$ was established for the Wasserstein-2 distance. In our first main theorem, we significantly extend this result in several directions. We allow for more general Bernoulli random variables $B_k \sim {\rm Ber}(\theta/(k+\beta))$ for $\theta>0$ and $\beta\in\mathbb{R}$, and consider $\frac{1}{n}\sum_{k=l}^n k B_k$ for a suitable $l\in\mathbb{N}$. For the Kolmogorov distance between such random variables and generalized Dickman distributed random variables, we derive upper bounds and establish the optimality of the rates of convergence. 
Throughout, we will denote a Dickman distributed random variable with parameter $\theta>0$ by $D_\theta$.

\begin{theorem}\label{thm.ber}
	For $\theta>0$, $\beta \in \R$ and $n, l \in \N$ with $n \ge l \ge \theta-\beta$, let $W_n= n^{-1} \sum_{k=l}^n kB_k$, where $B_k \sim {\rm Ber}(\theta/(k+\beta))$ are independent. Then, there exists a constant $C \in (0,\infty)$ depending only on $\theta$ and $\beta$ such that
\begin{equation}\label{eq:thm1}
	d_K(W_n,D_\theta) \le \begin{dcases} C \left(\frac{l + |\beta| \log(n/l)}{n}\right), & \quad \theta\geq 1,\\  \frac{C l^\theta}{n^\theta}, & \quad \theta \in (0,1). \end{dcases}
\end{equation}
Moreover, the bound is of optimal order in $n$, that is, there exists a constant $C' \in (0,\infty)$ depending only on $\theta, \beta$ and $l$ such that
$$
d_K(W_n,D_\theta) \ge \begin{dcases} C' \left(\frac{1 + |\beta| \log n}{n}\right), & \quad \theta\geq 1,\\  \frac{C'}{n^\theta}, & \quad \theta \in (0,1). \end{dcases}
$$
\end{theorem}

Note that the numerator $\theta$ of the success probabilities $\theta/(k+\beta)$ of the Bernoulli distributed random variables is the parameter of the limiting generalized Dickman distribution. The parameter $\beta$ in the denominator only affects the rate of convergence for $\theta\geq 1$. We obtain for $\beta=0$ the rate $1/n$ that was established in \cite{BG17} for the Wasserstein-2 distance. For $\beta\neq0$ the rate slows down due to the presence of an additional logarithmic factor. Our rates of convergence are of the optimal order in $n$ if we let $l$ be fixed. The lower bounds for Dickman approximation derived in this paper (see Proposition~\ref{prop:optimality}, also Remark~\ref{rem:opt}) allow even the situation that $l$ depends on $n$. Except for some small range in the case that $\theta\geq 1$ and $\beta<0$, they are of the same order in both $n$ and $l$ as in the upper bounds in \eqref{eq:thm1}.

In the following, we discuss two application of Theorem~\ref{thm.ber}. The first one is in the context of the runtime of the Quickselect algorithm to find the smallest number of a list of $n$ distinct numbers with $n \ge 2$. The algorithm works as follows: first choose a key $x$ from the $n$ numbers uniformly at random and then regroup the numbers into two groups corresponding to the elements whose values are less than and greater than $x$, respectively. Then we continue recursively by moving to the group left of $x$, or stop if $x$ is the desired smallest number.  Let $C_{n}$ be the number of comparisons made by the algorithm to find the smallest element of the list of $n$ distinct numbers. In \cite{Hwa}, it was shown that $C_n$, suitably normalized, converges in distribution to a Dickman distributed random variable. We refine this result by providing a bound of optimal order for the Kolmogorov distance, which is a consequence of our Theorem~\ref{thm.ber}. We note here that a bound of the same order was provided in \cite{Go17} in the Wasserstein distance using a completely different method employing the so-called Dickman bias coupling. We simply write $D$ for a standard Dickman random variable (i.e.\ the case $\theta=1$).

\begin{corollary}\label{cor:Q} There exists a universal constant $C \in (0,\infty)$ such that for any $n \in \N$ with $n\geq 2$,
	$$
	d_K(n^{-1} C_{n} - 1, D) \le \frac{C \log n}{n}.
	$$
	Moreover, the bound is of optimal order.
\end{corollary}

As our second application, we consider the weighted depth $W_{n}$ of the node $n$ in a randomly grown simple increasing tree of size $n \in \N$. By \cite[Section~2.2, Lemma~1]{KP07} (see also \cite[Lemma~5]{PP07}), each such tree can be constructed via a tree evolution process, which is described by intrinsic parameters $c_1 \in \R\setminus\{0\}$ and $c_2 \in \R$ with $c_2/c_1>-1$ and works as follows: At step 1, we start with the root labeled 1. For $i \in \N$, the step $i+1$ involves adding the new node $i+1$ to the tree by attaching it to a previous node $v$ with probability $p(v)$ that is a function of $i$, the intrinsic parameters $c_1$ and $c_2$, and the out-degree of $v$ at time $i$ (see \cite[Section~2.2, Lemma~1]{KP07} for exact formulae). The weighted depth $W_n$ of the node $n$ is now the sum of all labels on the path from $n$ to the root $1$, including the labels $1$ and $n$ (see \cite[Section~1]{KP07}). In \cite[Section 3, Theorem 1]{KP07}, it was shown that one obtains the Dickman distribution with parameter $1+c_2/c_1$ as the limiting distribution for the appropriately re-scaled weighted depth $W_n$. In the following corollary, we provide an optimal rate of convergence in this case.

\begin{corollary}\label{cor:IT} For $n \in \N$, let $W_{n}$ be the weighted depth of node $n$ in a randomly grown simple
	increasing tree of size $n$ and write $\beta:=c_2/c_1$. Then
	$$
	d_K\left(\frac{W_{n} - n}{n}, D_{1+\beta}\right) \le  \begin{dcases} C \left(\frac{1 + \beta \log n}{n}\right), & \quad \beta \geq 0,\\  \frac{C}{n^{1+\beta}}, & \quad \beta \in (-1,0), \end{dcases}
	$$
	for some constant $C\in (0,\infty)$ depending only on $\beta$. Moreover, the bound is of optimal order.
\end{corollary}

We also prove a result similar to Theorem~\ref{thm.ber} for weighted sums of Poisson random variables.
Below, ${\rm Poi}(\lambda)$ stands for the Poisson distribution with mean $\lambda>0$.

\begin{theorem}\label{thm:Poi}
For $\theta>0$, $\beta \in \R$ and $n, l \in \N$ with $n \ge l > -\beta$, let $W_n= n^{-1} \sum_{k=l}^n kP_k$, where $P_k \sim {\rm Poi}(\theta/(k+\beta))$ are independent. Then, there exists a constant $C \in (0,\infty)$ depending only on $\theta$ and $\beta$ such that
	\begin{equation*}
			d_K(W_n,D_\theta) \le \begin{dcases} C \left(\frac{l + |\beta| \log(n/l)}{n}\right), & \quad \theta\geq 1,\\  \frac{C l^\theta}{n^\theta}, & \quad \theta \in (0,1). \end{dcases}
	\end{equation*}
Moreover, the bound is of optimal order in $n$.
\end{theorem}

The rates of convergence here depend on $\theta$ and $\beta$ exactly as in Theorem~\ref{thm.ber}. For $\beta=0$ rates of convergence of the order $1/n$ for the Wasserstein-2 distance were derived in \cite[Theorem~1.2]{BG17}.

In addition, we also improve the bounds in \cite[Theorems~1.1 and 1.2]{BG17} for the Wasserstein-2 distance, and in the Bernoulli case, generalize the result to any $\theta>0$, when randomly weighted sums are considered. Recall that $d_{1,1}$ is defined as in \eqref{eq:gendis} with the class of test functions taken to be $\mathcal{H}={\mathcal{H}}_{1,1}$ defined at \eqref{eqn:TestFunctionsH11}.

\begin{theorem}\label{thm:imp}
	Let $\theta>0$ and let $(Y_k)_{k \in \N}$ be a sequence of independent non-negative random variables with $\E\,[Y_k]=k$ and ${\rm Var}(Y_k)=\sigma_k^2$ for all $k \in \N$.
	\begin{enumerate}[(a)]
		\item For $W_n=n^{-1}\sum_{k=\lceil \theta \rceil}^n Y_k B_k$, where $(B_k)_{k \in \N}$ are independent with $B_k \sim {\rm Ber}(\theta/k)$, and also independent of $(Y_k)_{k \in \N}$,
		$$
		d_{1,1}(W_n,D_\theta) \le \frac{\theta(6\theta +1)}{4n} + \frac{\theta}{2n^2} \sum_{k=\lceil \theta \rceil}^n\frac{\sigma_k^2}{k}, \quad n \in \N.
		$$
		\item  For $W_n=n^{-1}\sum_{k=1}^n Y_k P_k$, where $(P_k)_{k \in \N}$ are independent with $P_k \sim {\rm Poi}(\theta/k)$, and also independent of $(Y_k)_{k \in \N}$,
		$$
		d_{1,1}(W_n,D_\theta) \le \frac{\theta}{4n} + \frac{\theta(1+\theta)}{2n^2} \sum_{k=1}^n\frac{\sigma_k^2}{k}, \quad n \in \N.
		$$ 
	\end{enumerate}
\end{theorem}

\begin{remark}
In the setting of Theorem~\ref{thm:imp}(a), when $\theta=1$, Theorem~1.1 in \cite{BG17} provides the bound
$$
d_{1,1}(W_n,D) \le \frac{3}{4n} + \frac{1}{2n^2} \sum_{k=1}^n\frac{\sqrt{(\sigma_k^2 + k^2)\sigma_k^2}}{k}, \quad n \in \N.
$$
Thus, the second term of our bound in Theorem \ref{thm:imp}(a) is strictly smaller and is often of a better order. For instance, when $\sigma_k^2$ is of the order $k^{2-\eps}$ for some $\eps>0$, then the bound in \cite[Theorem~1.1]{BG17} is of the order $n^{-\min\{1,\eps/2\}}$ while our bound is of the order $n^{-\min\{1,\eps\}}$. Theorem~\ref{thm:imp}(b) improves upon \cite[Theorem~1.2]{BG17} in a similar way. 
\end{remark}

\begin{remark}
For normal approximation, it is often the case that Wasserstein type distances and the Kolmogorov distance have optimal bounds of the same order. Note that by Theorem~\ref{thm:imp}, one obtains an upper bound of the order $1/n$ on the Wasserstein-2 distance between $W_n=n^{-1}\sum_{k=\lceil \theta \rceil}^n k B_k$ or $W_n=n^{-1}\sum_{k=1}^n kP_k$ and $D_\theta$ for all $\theta>0$. But by Theorems~\ref{thm.ber} and \ref{thm:Poi}, for $\theta \in (0,1)$ the optimal bound in the Kolmogorov distance is in both cases of the order $1/n^\theta$, which is strictly larger.
\end{remark}

It remains an open question whether one can generalize the bounds in Theorems \ref{thm.ber} and \ref{thm:Poi} for the Kolmogorov distance to the situations with random weights as considered in Theorem~\ref{thm:imp}(a) and (b).

Results like Theorems~\ref{thm.ber} and \ref{thm:Poi} for the Kolmogorov distance are first of their kind in the context of Dickman convergences. Now we briefly discuss the approach we take to show these results. A very general technique to derive bounds for distributional approximations is to use Stein's method, as introduced in \cite{St72} and further developed in \cite{St86}. The first step in developing Stein's method for a particular target distribution is to specify a characterizing equation for the said distribution. The following result from \cite{Go17v1} gives such a characterization for the Dickman distributions.
For $\theta>0$ and a real-valued function $g$ on $\R_+$, define the averaging operator
\begin{equation}\label{def.A}
	\A_{\theta,x} g = \E\,[g(U^{1/\theta}x)]= \left\{
	\begin{array}{cl}
		g(0), \quad & \mbox{$x=0$,}\\
		\frac{\theta}{x^\theta} \int_0^x  g(u)u^{\theta-1}du, \quad & \mbox{$x>0$,}\end{array}\right. 
\end{equation}
where $U \sim \mathbb{U}[0,1]$. 

\begin{prop}[Lemma 3.2, \cite{Go17v1}]\label{prop.go}
	For $\theta>0$ and a non-negative random variable $W$, the following are equivalent:
	\begin{enumerate}[(a)]
		\item $W \sim D_\theta$.
		\item $\E\,[h(W)]=\E\,[\A_{\theta,W+1} h]$ for all $h$ for which $\E\,[h(W)]$ exists.
		\item $\E \,[(W/\theta)f'(W) - f(W+1) + f(W)]=0 \qm{for all $f \in \bigcup_{\alpha >0} {\rm Lip}_\alpha$}$.
	\end{enumerate}
\end{prop}

To apply Stein's method, one typically forms a Stein equation based on such a characterization. A Stein equation is usually a differential equation involving certain test functions in a large class corresponding to a desired probability distance. For an extensive treatment of the subject, see e.g.\ the book \cite{CGS} or the surveys \cite{Ro11,Ch14}.

We use the characterization $(c)$ in Proposition~\ref{prop.go} to form a Stein equation. For $\theta>0$, $\mathcal{H}$ an appropriate class of test functions and $h \in \mathcal{H}$, as in \cite{BG17}, we consider the Stein equation 
\begin{equation}\label{eq:fstein}
(x/\theta) f'(x)+f(x)-f(x+1)=h(x)-\E\, [h(D_{\theta})], \quad x \in \R_+.
\end{equation}
The next key step of the method is to prove that the solutions to the Stein equation over the given class of test functions are sufficiently smooth. This makes choosing the class $\mathcal{H}$ very crucial. Typically, the more smoothness the functions in $\mathcal{H}$ have, the smoother the solutions become. Unlike in \cite{BG17}, where the class of smooth functions $\mathcal{H}_{1,1}$ was considered, in this paper, we consider test functions in $\mathcal{H}_K$ defined at \eqref{eq:dK}, which are not everywhere smooth. In the following result, we prove the existence of a solution $f$ to the Stein equation \eqref{eq:fstein} when the test functions are in $\mathcal{H}_K$ and provide its smoothness properties. We denote by $\|\bdot\|_{\mathbb{R}_+}$ the supremum norm over $\mathbb{R}_+$ and write $\N_0:=\N \cup \{0\}$. For $k \in \N_0$, we let $ \A_{\theta,x+1}^k$ be the $k$-fold iteration of the operator $\A_{\theta,x+1}$, where we use the convention $\A_{\theta,x+1}^0 g=g(x)$ for $x\in \R_+$ and a real-valued function $g$ on $\R_+$.

\begin{theorem}\label{thm.bdkol}
Let $\theta>0$.
	\begin{enumerate}[(a)]
		\item For any $h \in \cup_{\alpha>0} {\rm Lip}_\alpha \cup \mathcal{H}_K$, the function series
\begin{equation} \label{def:hstar.g}
	g(x) :=\sum_{k \in \N_0} \A_{\theta,x+1}^k (h-\E \, [h(D_\theta)]), \quad x\in \R_+,
\end{equation}
is absolutely convergent.
\item For $a\in \R_+$, let $g_a$ be the function in \eqref{def:hstar.g} with $h=h_a(\bdot):=\mathds{1}(\bdot \le a)$. Then, $f(x)=\A_x g_a$ for $x\in \R_+$ solves the Stein equation \eqref{eq:fstein} on $\R_+ \setminus \{a\}$. Further, $f$ can be expressed as $f=f_1+f_2$ where $f_1(x)=\min\{1,a^\theta/x^\theta\} - \Prob{D_\theta \le a}$ and $f_2$ is a differentiable function such that $f_2'$ can be decomposed as $f_2'=u_{+} - u_{-}$ with $u_{+}$ and $u_{-}$ being non-negative and non-increasing, and satisfying
		\begin{equation}\label{bound:f2}
			\|u_{+}\|_{\mathbb{R}_+} \le \theta^2 \qmq{and} \|u_{-}\|_{\mathbb{R}_+} \le \theta^2.
		\end{equation}
	\end{enumerate}
\end{theorem}

It should be noted that Theorem~\ref{thm.bdkol}(a) was shown for Lipschitz functions in \cite[Theorem~4.1]{BG17}, in particular, it was proved that the function $g$ is Lipschitz. In Theorem~\ref{thm.bdkol}(b), for $h\in\mathcal{H}_K$ we provide a solution $f$ to the Stein equation \eqref{eq:fstein}, which splits into two parts $f_1$ and $f_2$, with $f_2$ having sufficient smoothness properties and a derivative that is a difference of two monotone functions satisfying \eqref{bound:f2}, while the other part $f_1$ is differentiable at all but one point. In the proofs of our main results, we study the contributions coming from $f_1$ and $f_2$ separately by using the explicit form of $f_1$ and the properties of $f_2'$, respectively.

The rest of the paper is organized as follows. In Section~\ref{sec.kol.stein}, we prove Theorem~\ref{thm.bdkol}. We apply Theorem~\ref{thm.bdkol} to provide a general upper bound for certain weighted Bernoulli sums in Theorem~\ref{thm.ber'} in Section~\ref{sec:upperbound}, which is then used to obtain the upper bounds in Theorems~\ref{thm.ber} and \ref{thm:Poi}. A general lower bound result for weighted Bernoulli and Poisson sums is proved in Proposition~\ref{prop:optimality} in Section~\ref{sec:lowerbound}, as consequence of which, we prove the optimality results in Theorems~\ref{thm.ber} and \ref{thm:Poi}. Finally, in Section~\ref{sec:mainresults}, we establish the Corollaries~\ref{cor:Q} and \ref{cor:IT}, and Theorem \ref{thm:imp}.

\section{Solution to the Stein equation}\label{sec.kol.stein} 
This section is dedicated to the proof of Theorem~\ref{thm.bdkol}. We will apply the following result from \cite{Go17v1}.

\begin{prop}[{\cite[Theorem 3.1]{Go17v1}}]\label{prop.go1}
For $\theta>0$ and $v \in {\rm Lip}_1$ with $\E\,[v (D_\theta)]=0$, the function
$$
w(x) = \sum_{k\in \N_0} \A_{\theta,x+1}^kv, \quad x\in \R_+,
$$
is a ${\rm Lip}_{1+\theta}$ solution to the equation
$$
w(x)-\A_{\theta,x+1}w=v(x), \qm{$x \in \R_+$}.
$$
\end{prop}

We will also make use of the fact (see e.g.\ \cite[Proposition~3(v)]{PW04}) that for $\theta>0$, one has $\E[D_\theta]=\theta$. We now present the proof of Theorem~\ref{thm.bdkol}.

\begin{proof}[Proof of Theorem~\ref{thm.bdkol}] 
Throughout this proof, we let $\widetilde h(x)=h(x)-\E\,[h(D_\theta)]$, $x\in \R_+$, denote the centered version of $h$. Since $\theta$ remains fixed, for convenience, we will write $\A_x$ for for the averaging operator $\A_{\theta,x}$ defined at \eqref{def.A}. For $k\in\N_0$, we define
$$
\widetilde h^{(\star k)}(x)=\A_{x+1}^k \widetilde h, \quad x\in \R_+,
$$ 
with the convention $\widetilde h^{(\star 0)}(x)=\A_{x+1}^0 \widetilde h=\widetilde h(x)$ for $x\in \R_+$.

We start by proving Theorem~\ref{thm.bdkol}(a). First we consider $h \in \cup_{\alpha>0} {\rm Lip}_\alpha$. Notice, it suffices to show the result for $h\in {\rm Lip}_1$. In this case, for $k\in \N$ and $x\in \R_+$, using the definition of $\widetilde h^{(\star k)}$ and \eqref{def.A} for the first step and repeating this $k$ times, we have that
\begin{align}\label{eq:tildeh}
\widetilde h^{(\star k)}(x)  &= \E \left[\widetilde h^{(\star (k-1))}(U_1^{1/\theta} (1+x))\right]= \mathbf{E} \left[ \widetilde{h}(U_1^{1/\theta} (1 + U_2^{1/\theta} (1+\cdots + U_k^{1/\theta} (1+x) )  )) \right]\\
& \qquad = \mathbf{E} \left[ h(U_1^{1/\theta} (1 + U_2^{1/\theta} (1+\cdots + U_k^{1/\theta} (1+x) )  )) \right] - \mathbf{E}\left[h(D_\theta)\right] \nonumber\\
& \qquad = \mathbf{E} \left[ h(U_1^{1/\theta} (1 + U_2^{1/\theta} (1+\cdots + U_k^{1/\theta} (1+x) )  )) \right] \nonumber\\
& \qquad \qquad\qquad \qquad - \mathbf{E} \left[ h(U_1^{1/\theta} (1 + U_2^{1/\theta} (1+\cdots + U_k^{1/\theta} (1+D_\theta) )  ))\right]\nonumber
\end{align}
with independent $U_1,\dots,U_k \sim \mathbb{U}[0,1]$ and $D_\theta$, where the last step is due to \eqref{def:Dickman}. Thus from the Lipschitz property of $h$ and that $\E[D_\theta]=\theta$, it follows that for $k \in \N$,
$$
|\widetilde h^{(\star k)}(x)| \leq \mathbf{E} \left[ U_1^{1/\theta}\cdots U_k^{1/\theta} |x-D_\theta| \right] \leq \bigg( \frac{\theta}{\theta+1} \bigg)^k (x+\mathbf{E}\left[D_\theta\right])=\bigg( \frac{\theta}{\theta+1} \bigg)^k (x+\theta).
$$
Summing, we obtain for $x \in\R_+$ that $\sum_{k  \in\N_0} |\widetilde h^{(\star k)}(x)| \le |\widetilde h(x)| + \theta(x+\theta)$, proving the absolute convergence of the series $g$ in \eqref{def:hstar.g}. 
	
Now let $h=\mathds{1}(\bdot \le a) \in \mathcal{H}_K$ with $a \in \R_+$. For $x \in \R_+$, we can write
\begin{equation}\label{eq:g}
g(x)=\widetilde h(x) + \sum_{k \in \N} \widetilde h^{(\star k)}(x) =\widetilde h (x) + \sum_{k \in \N_0} \A_{x+1}^{k}(\A_{x+1} \widetilde h).
\end{equation}
Moreover, by \eqref{def.A} we have
\begin{align*}
\A_{x+1}\widetilde h &=(x+1)^{-\theta}\int_0^{x+1}(\mathds{1}(u \le a) -\Prob{D_\theta \le a})\theta u^{\theta-1} du\\
&=\min\{1,a^\theta/(x+1)^\theta\}-\Prob{D_\theta \le a}.
\end{align*}
It is now easy to see that the function $\A_{x+1}\widetilde h$ is Lipschitz with Lipschitz constant at most $\theta$. Arguing as above, this yields that $\sum_{k \in \N} \A_{x+1}^{k}\widetilde h$ is absolutely convergent, proving the same for $g$.
	
Next we show Theorem~\ref{thm.bdkol}(b). For economy of notation, we fix $a \in \R_+$ and drop the subscript $a$ from $h_a$, $\widetilde h_a$ and $g_a$. Notice, $\E \,[\A_{D_\theta+1}\widetilde h]=0$ by Proposition \ref{prop.go}(b). Since $\A_{x+1}\widetilde h$ is Lipschitz with Lipschitz constant at most $\theta$, we have $\theta^{-1}\A_{x+1}\widetilde h \in {\rm Lip}_1$. Denoting $\overline{g} =g-\widetilde h$, notice by \eqref{eq:g} and Proposition \ref{prop.go1} that $\theta^{-1} \overline{g} \in {\rm Lip}_{1+\theta}$ and
	$$\overline{g}(x)- \A_{x+1} \overline{g} =\A_{x+1}\widetilde h, \quad x\in \R_+.$$
This implies
$$
g(x)- \A_{x+1}g = \overline{g}(x)- \A_{x+1}\overline{g} + \widetilde{h}(x)- \A_{x+1}\widetilde{h} = \A_{x+1}\widetilde{h} + \widetilde{h}(x)- \A_{x+1}\widetilde{h} = \widetilde{h}(x)
$$
for $x \in\R_+$. Since $g$ is continuous on $\R_+ \setminus\{a\}$, it is straightforward to check upon differentiation that $f(x)=\A_x g$ satisfies $f'(x)=(\theta/x) (g(x)- \A_xg)$ for $x\in \R_+ \setminus \{a\}$. Combining the two previous facts leads to 
$$
(x/\theta) f'(x)+ f(x)-f(x+1) = g(x) - \A_xg + \A_xg - \A_{x+1}g = g(x) - \A_{x+1}g =\widetilde{h}(x)
$$
for $x\in \R_+ \setminus\{a\}$, proving that $f$ solves the Stein equation \eqref{eq:fstein}.
		
For the second assertion in Theorem~\ref{thm.bdkol}(b), write $f(x)=\A_x g$ as
\begin{equation*}
f(x)=\A_x \sum_{k \in \N_0} \widetilde h^{(\star k)} = \sum_{k \in \N_0} \A_x \widetilde h^{(\star k)}=\A_x \widetilde h + \sum_{k \in \N} \A_x \widetilde h^{(\star k)}=: f_1(x) + f_2(x),
\end{equation*}
where the second equality (exchanging $\A_x$ and the sum) follows by a standard argument where one considers partial sums over compact sets and takes a limit (see e.g.\ the proof of \cite[Theorem 4.9]{BG17}). Evaluating $f_1$ and taking the derivative, we obtain
\begin{equation*}
f_1(x)=\min\{1,a^\theta/x^\theta\} - \Prob{D_\theta \le a}, \,x \in \R_+, \qmq{and} f_1'(x)=-\theta a^\theta\mathds{1}(x>a)/x^{\theta+1}, \, x \in \R_+\setminus\{a\}.
\end{equation*}
For $(U_i)_{i \in \N}$ a sequence of i.i.d.\ $\mathbb{U}[0,1]$-distributed random variables, for $k \in \N$, by \eqref{eq:tildeh} we have
\begin{equation*}
\widetilde h^{(\star k)}(x)=A_{x+1}^k \widetilde h = \E\left[\widetilde h\left(U_1^{1/\theta}\left(1+U_2^{1/\theta}\left(1+\cdots + U_k^{1/\theta}(1+x)\right)\right)\right)\right].
\end{equation*}
Hence for $U \sim \mathbb{U}[0,1]$ independent of $(U_i)_{i \in \N}$, by \eqref{def.A} we obtain 
\begin{align*}
A_x \widetilde h^{(\star k)}=&\E\left[\widetilde h\left(U_1^{1/\theta}\left(1+U_2^{1/\theta}\left(1+\cdots + U_k^{1/\theta}(1+U^{1/\theta}x)\right)\right)\right)\right]\nonumber\\
=&\E \left[h\left(U_1^{1/\theta}\left(1+U_2^{1/\theta}\left(1+\cdots + U_k^{1/\theta}(1+U^{1/\theta}x)\right)\right)\right)\right] - \E\, [h(D_\theta)]\nonumber\\
=&\E\left[ \E \left[\mathds{1}\left\{U_1^{1/\theta} \le \frac{a}{1+U_2^{1/\theta}(1+\cdots + U_k^{1/\theta}(1+U^{1/\theta}x))}\right\} \Bigg| U,(U_i)_{i \ge 2}\right]\right]-\E\, [h(D_\theta)] \nonumber\\
=&\E\left[\min\left\{1, \left(\frac{a}{1+U_2^{1/\theta}\left(1+\cdots + U_k^{1/\theta}(1+U^{1/\theta}x)\right)}\right)^\theta\right\}\right]-\E\, [h(D_\theta)].
\end{align*}
For notational simplicity, let us denote
$$m_k(x)=\frac{a}{1+U_2^{1/\theta}\left(1+\cdots + U_k^{1/\theta}(1+U^{1/\theta}x)\right)}$$
for $k\in \N$, where we note that $m_1(x)=a/(1+U^{1/\theta}x)$. Noting that the expectation and derivative can be interchanged, we have for each $k \in \N$,
\begin{align}\label{eq:der}
\frac{d}{dx}A_x \widetilde h^{(\star k)}
& = -\E\left[\mathds{1}(m_k(x)<1) \frac{\theta m_k(x)^\theta(UU_2\dots U_k)^{1/\theta}}{1+U_2^{1/\theta}\left(1+\cdots + U_k^{1/\theta}(1+U^{1/\theta}x)\right)}\right] \nonumber\\
& = -\E\left[ \frac{(\mathds{1}(m_k(x)<1) m_k(x)^\theta + \mathds{1}(m_k(x)\geq 1) ) \theta (UU_2\dots U_k)^{1/\theta}}{1+U_2^{1/\theta}\left(1+\cdots + U_k^{1/\theta}(1+U^{1/\theta}x)\right)}\right] \nonumber\\
& \quad +\E\left[ \frac{\mathds{1}(m_k(x)\geq 1)  \theta (UU_2\dots U_k)^{1/\theta}}{1+U_2^{1/\theta}\left(1+\cdots + U_k^{1/\theta}(1+U^{1/\theta}x)\right)}\right] = : -s_k^{-}(x) +  s_k^{+}(x).
\end{align}
For $k\in \N$ we have
$$
\E\left[ (UU_2\dots U_k)^{1/\theta} \right] = \E\left[ U^{1/\theta} \right]^k \leq \left(\frac{\theta}{\theta+1}\right)^{k}
$$
so that for $x \in \R_+$,
$$
|s_k^{-}(x)| \leq \theta \left(\frac{\theta}{\theta+1}\right)^{k} \quad \text{and} \quad |s_k^{+}(x)| \leq \theta \left(\frac{\theta}{\theta+1}\right)^{k}.
$$
Since $f_2(x)=\sum_{k \in \N} \A_x \widetilde h^{(\star k)}$, this and \eqref{eq:der} imply
$$
f_2'(x) = -\sum_{k\in \N} s_k^{-}(x) + \sum_{k\in \N} s_k^{+}(x) =: -u_{-}(x) + u_{+}(x) 
$$
with
$$
\|u_{-}\|_{\R_+}\leq \theta^2 \quad \text{and} \quad \|u_{+}\|_{\R_+}\leq \theta^2.
$$
As $s_k^{-}$ and $s_k^{+}$ are non-negative and non-increasing for all $k\in \N$, we see that $u_{-}$ and $u_{+}$ are also non-negative and non-increasing, proving the desired conclusion.
\end{proof}

\section{Upper bounds}\label{sec:upperbound}
The goal of this section is to prove Theorem~\ref{thm.ber'} below, which we use to derive the upper bounds in Theorems \ref{thm.ber} and \ref{thm:Poi} at the end of this section. For $\theta>0$ and $p_\theta$ the density of the generalized Dickman distribution with parameter $\theta$, one can show (see e.g.\ \cite[Equation (1.1)]{Pi16b}) that $p_\theta=(e^{-\theta \gamma}/\Gamma(\theta)) \varrho_\theta$ where $\Gamma(\bdot)$ is the Gamma function, $\gamma$ is the Euler-Mascheroni constant, and $\varrho_\theta$ satisfies
\begin{align}
&\varrho_{\theta} (x) = 0, \quad x \leq 0, \nonumber\\ 
&\varrho_{\theta} (x) = x^{\theta - 1} ,\quad 0 < x \leq 1, \label{eq:rho}\\ 
&x \varrho_{\theta}'(x) + (1 - \theta) \varrho_{\theta} (x) + \theta \varrho_{\theta} (x - 1) = 0 ,\quad  x > 1.\nonumber 
\end{align}
For $x\in[0,1]$ it follows from \eqref{eq:rho} that
\begin{equation}\label{eqn:P_D_theta}
\Prob{D_\theta \le x}=\frac{e^{-\theta \gamma}}{\Gamma(\theta)}\int_0^x t^{\theta-1} dt=\frac{e^{-\theta \gamma}x^\theta}{\Gamma(\theta) \theta}=\frac{e^{-\theta \gamma}x^\theta}{\Gamma(\theta+1)}.
\end{equation}
Below, we employ the fact (see e.g.\ \cite[Section~1]{Pi16b}) that there exists a constant $k_\theta>0$ such that 
\begin{equation}\label{eq:pbd}
p_\theta(x) \le \frac{k_\theta}{\Gamma(x+1)} \le k_\theta \,\text{ for } \; x \ge 1,
\end{equation} 
and in particular, we can take $k_1=e^{-\gamma}$. Note that for $\theta \ge 1$, by \eqref{eq:rho} and \eqref{eq:pbd}, we have
\begin{equation}\label{eq:pub}
	\|p_\theta\|_{\R_+} \le \max \left\{\frac{e^{-\theta \gamma}}{\Gamma(\theta)},k_\theta\right\} =: K_\theta.
\end{equation}
For $\theta>0$ and $s \in \R_+$, denote by $R_\theta(s)$ the function
\begin{equation*}
R_\theta(s) = \begin{cases} K_\theta, & \quad  \theta \ge 1,\\  s^{\theta-1}, & \quad \theta \in (0,1). \end{cases}
\end{equation*}
In the following, we will use the convention that $R_\theta(0) \cdot 0 =0$. Before presenting the proof of Theorem~\ref{thm.ber'}, we first prove a few preliminary lemmas.

\begin{lemma}\label{lem:bound_Dickman}
For $\theta>0$, $\alpha > 0$, $u \ge 0$ and $m\in\N_0$,
$$
\Prob{m \alpha \le D_\theta + u\le (m+1)\alpha} \le R_\theta(\alpha) \alpha.
$$
\end{lemma}

\begin{proof}
By \eqref{eq:pub}, for $\theta\geq 1$, $\alpha>0$, $u \ge 0$ and $m \in \N_0$,
	$$\Prob{m \alpha \le D_\theta + u \le (m+1)\alpha} \le K_\theta \alpha.$$
	For $\theta\in(0,1)$, from \eqref{eq:rho}, one can see that $\varrho_\theta$ is decreasing on $(0,\infty)$; indeed, it is trivially so on $(0,1]$ and for $x>1$ one has that $\varrho_{\theta}'(x)=  (\theta-1) \varrho_{\theta} (x)/x - (\theta/x) \varrho_{\theta} (x - 1) < 0$. Hence for $\alpha \in (0,1]$, $u \ge 0$ and $m \in \N_0$, we obtain
$$
\Prob{m \alpha \le D_\theta + u \le (m+1)\alpha} \le \Prob{0 \le D_\theta \le \alpha} = \frac{e^{-\theta\gamma}}{\Gamma(\theta+1)}\alpha^\theta \le \alpha^\theta,
$$
where we used \eqref{eqn:P_D_theta} and that $e^{-\theta\gamma}/ \Gamma(\theta+1) = \Prob{D_\theta \le 1} \le 1$. This is also trivially true for $\alpha >1$. Combining the inequalities for the cases $\theta \ge 1$ and $\theta \in (0,1)$ proves the assertion.
\end{proof}

For $\theta>0$ and independent non-negative random variables $S$ and $\Delta$, it follows from Lemma~\ref{lem:bound_Dickman} upon conditioning on $\Delta$ that
	\begin{align}
	d_K(S+\Delta,D_\theta) &\le \sup_{t \in \R} |\Prob{S \le t-\Delta} - \Prob{D_\theta \le t-\Delta}| + \sup_{t \in \R} |\Prob{D_\theta \le t-\Delta} - \Prob{D_\theta \le t}| \nonumber\\
		& \le d_K(S,D_\theta) + \E R_\theta(\Delta)\Delta \le d_K(S,D_\theta) + \widetilde{R} \label{eqn:d_K_approximation}
	\end{align}
with
$$
\widetilde{R}:= \begin{cases} K_\theta \mathbf{E}\Delta, & \quad \theta\ge 1, \\ (\mathbf{E}\Delta)^\theta, & \quad \theta\in(0,1), \end{cases}
$$
where we assumed independence of $\Delta$ and $D_\theta$, and used Jensen's inequality for $\theta\in(0,1)$.

The following lemma shows that for the Dickman approximation in Kolmogorov distance, it is enough to consider test functions $h(\bdot)=\mathds{1}(\bdot \le a)$ with $a \ge a_0$ for some $a_0\geq 0$ if one adds an error term that depends only on $a_0$.

\begin{lemma}\label{lem:lb}
For $\theta>0$ and a non-negative random variable $W$, assume that there exist $a_0,b\geq 0$ such that
\begin{equation}\label{eqn:Assumption_u_large}
\sup_{a\geq a_0} |\Prob{W \leq a} - \Prob{D_\theta\leq a}|\leq b.
\end{equation}
Then
$$
d_K(W,D_\theta) \leq R_\theta(a_0)a_0 + b.
$$
\end{lemma}

\begin{proof}
For $a \in [0, a_0)$, we have
\begin{align*}
|\mathbf{P}\{W \leq a\} - \mathbf{P}\{D_\theta\leq a\}| & \leq \max\{\mathbf{P}\{W \leq a\}, \mathbf{P}\{D_\theta\leq a\} \}\leq \max\{\mathbf{P}\{W \leq a_0\}, \mathbf{P}\{D_\theta\leq a_0\}\}\\
& \leq \mathbf{P}\{D_\theta\leq a_0\} + \big| \mathbf{P}\{W \leq a_0\}) - \mathbf{P}\{D_\theta\leq a_0\} \big|
\\
& \leq  \mathbf{P}\{D_\theta\leq a_0\} + b \leq R_\theta(a_0) a_0 + b,
\end{align*}
where in the penultimate and the last step, we have used \eqref{eqn:Assumption_u_large} and Lemma \ref{lem:bound_Dickman} with $\alpha=a_0$, $u=0$ and $m=0$, respectively. Combining this with \eqref{eqn:Assumption_u_large} for $a\geq a_0$ completes the proof.
\end{proof}

In the sequel, we will often use that for a non-negative random variable $W$ and $0 \le a <b <\infty$, 
\begin{equation}\label{eq:Kdis}
	 \Prob{a \le W \le b} \le \Prob{a \le D_\theta \le b} + 2 d_K(W, D_\theta).
\end{equation}

\begin{lemma}\label{lem:mom}
Let $\theta>0$ and let $W$ be a non-negative random variable.
\begin{enumerate}[(a)]
\item [(a)] For $\alpha>0$ and $u \ge 0$,
$$
\Prob{W + u \leq \alpha} \leq \left[R_\theta(\alpha)\alpha+d_K(W,D_\theta) \right] \mathds{1}(\alpha \ge u).
$$
\item [(b)] For $t>0$, $\alpha>0$ and $u \ge 0$,
\begin{equation*}
\E \left[\frac{\alpha^t}{(W+u)^{1+t}}\mathds{1}(W +u>\alpha)\right] \le \zeta(1+t)	\left( R_\theta (\max\{\alpha, u\}) + \frac{2d_K(W,D_\theta)}{\max\{\alpha, u\}}\right),
\end{equation*}
where $\zeta(k)=\sum_{m=1}^\infty m^{-k}$ for $k >1$ is the Riemann zeta function.
\end{enumerate}
\end{lemma}

\begin{proof}
For $\alpha>0$ and $u \ge 0$, from Lemma \ref{lem:bound_Dickman} with $m=0$, we deduce
$$
\Prob{W +u \leq \alpha} \leq  \left[\Prob{D_\theta + u \leq \alpha} + d_K(W,D_\theta)\right] \mathds{1}(\alpha \ge u) \le \left[R_\theta(\alpha) \alpha + d_K(W,D_\theta)\right] \mathds{1}(\alpha \ge u),
$$
which is the assertion in (a).

Next fix $\alpha>0$ and $u \ge 0$, and let $v=\max\{\alpha, u\}$. For $t>0$ using \eqref{eq:Kdis} and Lemma \ref{lem:bound_Dickman} for the third and the fourth steps, respectively, it follows that
\begin{align*}
\E \left[\frac{\alpha^t}{(W+u)^{1+t}}\mathds{1}(W + u>\alpha)\right] \le &\sum_{m=1}^\infty \E \left[\frac{\alpha^t}{(W+u)^{1+t}}\mathds{1}(m v\le W + u \le (m+1)v)\right]\\
\le &\sum_{m=1}^\infty \frac{\alpha^t}{m^{1+t} v^{1+t}} \Prob{mv\le W + u \le (m+1)v}\\
\le & \sum_{m=1}^\infty \frac{\alpha^t}{m^{1+t} v^{1+t}}(\Prob{mv \le D_\theta +u \leq (m+1)v} + 2d_K(W,D_\theta)) \\
\le &\sum_{m=1}^\infty \frac{1}{m^{1+t}v}(R_\theta(v)v +2d_K(W,D_\theta)) \\
= & \zeta(1+t)\left(R_\theta(v) + \frac{2d_K(W,D_\theta)}{v}\right),
\end{align*}
which proves the assertion (b).
\end{proof}

We will deduce the upper bounds for the Kolmogorov distances in Theorem \ref{thm.ber} and Theorem \ref{thm:Poi} from the following more general upper bound for weighted Bernoulli sums, which is the main result of this section. For $n \in \N$, let $\mathcal{R}_n$ denote the set of non-negative integer multiples of $1/n$, i.e.,
\begin{equation}\label{def:Rn}
	\mathcal{R}_n =\{ k/n : k \in \N_0\}.
\end{equation} 

\begin{theorem}\label{thm.ber'}
Let $l \in \mathbb{N}$, $\theta>0$ and $\tau \ge 0$ be fixed. Let $(Z_k)_{k\geq l}$ be independent random variables such that $Z_k=\sum_{i=1}^{m_k} B_k^{(i)}$ with independent $B_k^{(i)} \sim {\rm Ber}(p_k^{(i)})$, $i\in\{1,\hdots,m_k\}$, for some $m_k \in \N$ and $p_k^{(i)} \in [0,1]$, satisfying $\left|\sum_{i=1}^{m_k} p_k^{(i)} - \theta/k\right|\le \tau/k^2$. Define $W_n= n^{-1} \sum_{k=l}^n k Z_k$ for $n\geq l$. Then, there exists a constant $C \in (0,\infty)$ depending only on $\theta$ and $\tau$ such that for $n\geq l$,
$$
d_K(W_n,D_\theta) \le \begin{dcases} C \bigg(\frac{l + \tau \log (n/l)}{n}\bigg), & \quad \theta\geq 1,\\  \frac{C l^\theta}{n^\theta}, & \quad \theta \in (0,1). \end{dcases}
$$
\end{theorem}

\begin{proof}
The assumptions imply that
$$
\sum_{i=1}^{m_k} p_k^{(i)} \leq \frac{\theta}{k} + \frac{\tau}{k^2}
$$
for all $k\geq l$. Let $k_0\in\mathbb{N}$ be the smallest number depending only on $\theta$ and $\tau$ such that
$$
\frac{\theta}{k} + \frac{\tau}{k^2} \leq \frac{1}{2}
$$
for all $k\geq k_0$. In the following we first assume that $l\geq k_0$ whence $p_k^{(i)}\leq \frac{1}{2}$ for all $i\in\{1,\hdots,m_k\}$ and $k\geq l$. The case when $l < k_0$ is discussed at the end of the proof.

Fix $l \ge k_0$ and for $n\geq l$, define
$$
M_n= \begin{cases} 4 c_\theta \big(2C_1+l+ \frac{\tau}{\theta}\big(1+ \log(n/l)\big)\big)/n, & \quad \theta\geq 1, \\ \max\{6(2C_1+l)c_\theta, 36 \tau^2 \zeta(\theta+1)^2 \zeta(3/2)^2\}/n, & \quad \theta\in(0,1), \end{cases}
$$
with $c_\theta,C_1>0$ defined at \eqref{eqn:c_theta} and \eqref{eqn:C_1} below. With $\mathcal{R}_n$ as in \eqref{def:Rn}, we consider the case when $a \in [M_n,\infty)\setminus \mathcal{R}_n$.

Let $f$ be the solution to the Stein equation \eqref{eq:fstein} for $h=\mathds{1}(\bdot \leq a)$ from Theorem \ref{thm.bdkol}. Letting $W_n^{(k,i)}=W_n-(kB_k^{(i)})/n$, since $a \notin \mathcal{R}_n$, for $k \in \{l,\hdots,n\}$ and $i\in\{1,\hdots,m_k\}$ we have
\begin{equation*}
\E\left[ f'\left(W_n+\frac{k}{n}\right)\right] =\left(1-p_k^{(i)}\right) \E\left[f'\left(W_n^{(k,i)}+\frac{k}{n}\right)\right] + p_k^{(i)}\E \left[f'\left(W_n^{(k,i)}+\frac{2k}{n}\right)\right].
\end{equation*}
Since $p_k^{(i)} \leq \frac{1}{2}$, we derive
\begin{multline}\label{eq:Wk}
\E \left[f'\left(W_n^{(k,i)}+\frac{k}{n}\right)\right]=\frac{1}{1-p_k^{(i)}}\E \left[ f'\left(W_n+\frac{k}{n}\right)\right]-\frac{p_k^{(i)}}{1-p_k^{(i)}}\E \left[ f'\left(W_n^{(k,i)}+\frac{2k}{n}\right)\right]\\
=\E \left[f'\left(W_n+\frac{k}{n}\right)\right] +\frac{p_k^{(i)}}{1-p_k^{(i)}} \left(\E \left[f'\left(W_n+\frac{k}{n}\right)\right]-\E \left[f'\left(W_n^{(k,i)}+\frac{2k}{n}\right)\right]\right).
\end{multline}
For $U \sim \mathbb{U}[0,1]$ independent of $W_n$, it follows from Theorem \ref{thm.bdkol} that
\begin{align*}
\mathbf{P}\{W_n\leq a\} - \mathbf{P}\{D_\theta\leq a\} &=\E \left[(W_n/\theta) f'(W_n)-f'(W_n+U)\right]\\
&=\E\left[\sum_{k=l}^n \frac{k}{ \theta n} \sum_{i=1}^{m_k}B_k^{(i)} f'\left(W_n^{(k,i)}+\frac{k}{n}B_k^{(i)}\right)-f'(W_n+U) \right]\\
&=\frac{1}{n}\sum_{k=l}^n \E\left[\sum_{i=1}^{m_k} \frac{k p_k^{(i)}}{ \theta} f'\left(W_n^{(k,i)}+\frac{k}{n}\right)-f'\left(W_n+\frac{k-1}{n}+\frac{U}{n}\right) \right]\\
& \qquad \qquad- \frac{1}{n}\sum_{k=1}^{l-1} \E\left[f'\left(W_n+\frac{k-1}{n}+\frac{U}{n}\right) \right],
\end{align*}
where in the last step, we have taken expectation with respect to $B_k^{(i)}$ for the first term and conditioned on $U$ lying in the interval $[(k-1)/n,k/n)$ for the second and the third terms. Now using \eqref{eq:Wk} and noting that $\big|\sum_{i=1}^{m_k} k p_k^{(i)}/ \theta - 1 \big| \le \tau/(\theta k)$, we obtain
\begin{align}\label{eq:Tsp}
&\left|\mathbf{P}\{W_n\leq a\} - \mathbf{P}\{D_\theta\leq a\} \right| \nonumber\\
&\le \frac{1}{n} \Bigg|\sum_{k=l}^n \E\left[f'\left(W_n+\frac{k}{n}\right)-f'\left(W_n+\frac{k-1}{n}+\frac{U}{n}\right) \right] \Bigg|\nonumber\\
&\qquad + \frac{1}{n}\Bigg|\sum_{k=l}^n \sum_{i=1}^{m_k} \frac{k p_k^{(i)}}{\theta}  \frac{p_k^{(i)}}{1-p_k^{(i)}} \left(\E \left[f'\left(W_n+\frac{k}{n}\right)\right]-\E \left[f'\left(W_n^{(k,i)}+\frac{2k}{n}\right)\right]\right)\Bigg| \nonumber \\
&\qquad +\frac{1}{n}\Bigg|\sum_{k=1}^{l-1} \E\left[f'\left(W_n+\frac{k-1}{n}+\frac{U}{n}\right) \right]\Bigg| + \frac{\tau}{\theta n} \sum_{k=l}^n \frac{1}{k}\E\left[\bigg|f'\left(W_n+\frac{k}{n}\right)\bigg|  \right] \nonumber\\
& =: |T_1|+|T_2|+|T_3| + T_4. 
\end{align}
For $i\in\{1,2,3\}$ we rewrite $T_i$ as $T_i=T_{i,1}+T_{i,2,+}-T_{i,2,-}$ by using the decomposition $f'=f_1'+u_{+}-u_{-}$ from Theorem \ref{thm.bdkol}, i.e., we obtain $T_{i,1}$, $T_{i,2,+}$ and $T_{i,2,-}$ by replacing $f'$ by $f_1'$, $u_{+}$ and $u_{-}$, respectively, in the definition of $T_i$. For analogously defined $T_{4,1}$, $T_{4,2,+}$ and $T_{4,2,-}$, we have
$$
T_4 \leq T_{4,1}+T_{4,2,+}+T_{4,2,-}
$$
due to the absolute value in the definition of $T_4$.

We start by bounding $|T_{1,1}|$. Since $f'_1$ is $0$ in $[0,a)$ and increasing in $(a,\infty)$, we have a.s.,
$$
	\Bigg| \sum_{k=l}^n \left(f'_1\left(W_n+\frac{k}{n}\right)-f'_1\left(W_n+\frac{k-1}{n}+\frac{U}{n}\right)\right) \Bigg| \le \sup_{x\in\mathbb{R}_+\setminus\{a\}} |f_1'(x)| = \frac{\theta}{a},
	$$
	while if $W_n >a$,
	$$
	\Bigg| \sum_{k=l}^n \left(f'_1\left(W_n+\frac{k}{n}\right)-f'_1\left(W_n+\frac{k-1}{n}+\frac{U}{n}\right)\right) \Bigg|  \le |f_1'(W_n)|= \frac{\theta a^\theta}{W_n^{\theta+1}}.
	$$
	Thus, using Lemma~\ref{lem:mom} with $u=0$, we obtain
	\begin{align}\label{eq:T2temp}
	|T_{1,1}| &\le \frac{1}{n}\E \Bigg[\Bigg| \sum_{k=l}^n \left(f'_1\left(W_n+\frac{k}{n}\right)-f'_1\left(W_n+\frac{k-1}{n}+\frac{U}{n}\right)\right) \Bigg| \left(\mathds{1}(W_n \le a)+\mathds{1}(W_n > a)\right)\Bigg]\nonumber\\
	&\le \frac{1}{n}\E\left[\mathds{1}(W_n\le a) \frac{\theta}{a}\right]  + \frac{1}{n}\E\left[\mathds{1}(W_n> a) \frac{\theta a^\theta}{W_n^{\theta+1}}\right]\nonumber\\
	& \le \frac{\theta R_{\theta}(a)}{n} + \frac{\theta d_K(W_n,D_\theta)}{an} +\theta \zeta(\theta+1)\left(\frac{R_\theta(a)}{n} + \frac{2d_K(W_n,D_\theta)}{an}\right).
	\end{align}
	From the the properties of $u_{+}$ and $u_{-}$ it follows that
$$	
0 \leq -T_{1,2,+} \leq \frac{\theta^2}{n} \quad \text{and} \quad 0 \leq -T_{1,2,-} \leq \frac{\theta^2}{n}
$$
whence
\begin{equation}\label{eq:T2}
|T_1| \leq c_\theta \left(\frac{R_\theta(a)}{n}+ \frac{2d_K(W_n,D_\theta)}{an}\right) + \frac{\theta^2}{n},
\end{equation}
where, for the economy of notation, we write
\begin{equation}\label{eqn:c_theta}
	c_\theta=\theta(1+\zeta(\theta+1)).
\end{equation}

In order to bound $|T_2|$, we first prove an inequality, which is applied several times. Let $\widehat{f}$ be a non-negative and non-increasing function on $\mathbb{R}_+$. For ease of notation, let us write $s_{k,i}=\frac{k p_k^{(i)}}{\theta}  \frac{p_k^{(i)}}{1-p_k^{(i)}}$ for $k\geq l$ and $i\in\{1,\hdots,m_k\}$. Then, we have
\begin{align*}
& \frac{1}{n} \sum_{k=l}^n \sum_{i=1}^{m_k} s_{k,i} \E\left[ \widehat{f}\left( W_n + \frac{k}{n} \right) - \widehat{f}\left( W_n^{(k,i)} + \frac{2k}{n} \right) \right] \\
& \leq \frac{1}{n} \sum_{k=l}^n \sum_{i=1}^{m_k} s_{k,i} \E\left[ \widehat{f}\left( W_n + \frac{k}{n} \right) - \widehat{f}\left( W_n + \frac{2k}{n} \right) \right] \\
& = \frac{1}{n} \sum_{k=l}^n \sum_{i=1}^{m_k} s_{k,i} \sum_{j=k}^{2k-1} \E\left[ \widehat{f}\left( W_n + \frac{j}{n} \right) - \widehat{f}\left( W_n + \frac{j+1}{n} \right) \right] \\
& = \frac{1}{n} \sum_{j=l}^{2n-1} \E\left[ \widehat{f}\left( W_n + \frac{j}{n} \right) - \widehat{f}\left( W_n + \frac{j+1}{n} \right) \right] \sum_{k=l}^n \mathds{1}\left( \frac{j+1}{2} \leq k \leq j\right) \sum_{i=1}^{m_k} s_{k,i}.
\end{align*}
Using that $1-p_k^{(i)} \ge \frac{1}{2}$ and $\sum_{i=1}^{m_k} k p_k^{(i)} \le \theta+ \tau/k \le \theta+ \tau$ for $k\ge l$ yields
\begin{align*}
\sum_{k=l}^n \mathds{1}\left( \frac{j+1}{2} \leq k \leq j\right) \sum_{i=1}^{m_k} s_{k,i} & = \frac{1}{\theta}\sum_{k=l}^n \frac{\mathds{1}\left( \frac{j+1}{2} \leq k \leq j\right)}{k} \sum_{i=1}^{m_k} \frac{(kp_k^{(i)})^2}{1-p_k^{(i)}}\\
& \le \frac{2}{\theta} \sum_{k=l}^n \frac{\mathds{1}\left( \frac{j+1}{2} \leq k \leq j\right)}{k} \left(\sum_{i=1}^{m_k} kp_k^{(i)} \right)^2\\
& \le \frac{2(\theta+\tau)^2}{\theta} \sum_{k=l}^n \frac{\mathds{1}\left( \frac{j+1}{2} \leq k \leq j\right)}{k}.
\end{align*}
Noting that
\begin{multline*}
\sum_{k=l}^n \frac{\mathds{1}\left( \frac{j+1}{2} \leq k \leq j\right)}{k} \leq 1+ \int_{\frac{j+1}{2}}^j \frac{1}{x} dx
= 1+ \left( \ln(j) - \ln\left( \frac{j+1}{2}\right) \right)= 1+ \ln\left( \frac{2j}{j+1} \right) \leq 2,
\end{multline*} 
we obtain
\begin{equation}\label{eqn:C_1}
\sum_{k=l}^n \mathds{1}\left( \frac{j+1}{2} \leq k \leq j\right) \sum_{i=1}^{m_k} s_{k,i} \le \frac{ 4(\theta+\tau)^2}{\theta} =: C_1.
\end{equation}
Thus, we have
\begin{align}	\label{eq:InequalityHarmonic}
0 & \le  \frac{1}{n} \sum_{k=l}^n \sum_{i=1}^{m_k} s_{k,i}  \E\left[ \widehat{f}\left( W_n + \frac{k}{n} \right) - \widehat{f}\left(W_n^{(k,i)} + \frac{2k}{n} \right) \right] \nonumber\\
& \leq \frac{C_1}{n} \sum_{j=l}^{2n-1} \E\left[ \widehat{f}\left( W_n + \frac{j}{n} \right) - \widehat{f}\left( W_n + \frac{j+1}{n} \right) \right] \nonumber\\
&= \frac{C_1}{n} \E\left[ \widehat{f}\left( W_n + \frac{l}{n} \right) - \widehat{f}\left( W_n + 2 \right) \right]
\le \frac{C_1}{n} \E\left[ \widehat{f}\left( W_n + \frac{l}{n} \right)\right]. 
\end{align}

Since, for $x\in\R_+\setminus\{a\}$,
$$
f_1'(x) = -\frac{\theta a^\theta}{x^{\theta+1}} \mathds{1}(x>a)= \frac{\theta}{a}\mathds{1}(x\leq a) - \left(\frac{\theta}{a}\mathds{1}(x\leq a) + \frac{\theta a^\theta}{x^{\theta+1}} \mathds{1}(x>a) \right) =: v_{+}(x) -v_{-}(x)
$$
with $v_{+}$ and $v_{-}$ non-increasing and non-negative, it follows from \eqref{eq:InequalityHarmonic} that
\begin{align*}
|T_{2,1}| & \leq \frac{1}{n} \bigg|\sum_{k=l}^n \sum_{i=1}^{m_k} s_{k,i} \E\left[ v_{+}\left( W_n + \frac{k}{n} \right) - v_{+}\left(W_n^{(k,i)} + \frac{2k}{n} \right) \right] \bigg| \\
& \quad + \frac{1}{n} \bigg|\sum_{k=l}^n \sum_{i=1}^{m_k} s_{k,i} \E\left[ v_{-}\left( W_n + \frac{k}{n} \right) - v_{-}\left(W_n^{(k,i)} + \frac{2k}{n} \right) \right] \bigg|
 \\
& \leq \frac{C_1}{n} \E\left[ v_{+}\left( W_n + \frac{l}{n} \right) + v_{-}\left( W_n + \frac{l}{n} \right) \right] \leq \frac{C_1}{n} \E\left[ v_{+}\left( W_n \right) + v_{-}\left( W_n\right) \right] \\
& = \frac{C_1}{n} \E\left[ \frac{2\theta}{a} \mathds{1}(W_n\leq a) + \frac{\theta a^\theta}{W_n^{\theta+1}} \mathds{1}(W_n>a) \right].
\end{align*}
As in \eqref{eq:T2temp}, using Lemma~\ref{lem:mom} with $u=0$ we obtain
$$
|T_{2,1}| \leq C_1 \left[\frac{2\theta R_{\theta}(a)}{n} + \frac{2\theta d_K(W_n,D_\theta)}{an} +\theta \zeta(\theta+1)\left(\frac{R_\theta(a)}{n} + \frac{2d_K(W_n,D_\theta)}{an}\right)\right].
$$
Combining \eqref{eq:InequalityHarmonic} with the fact that $u_{+}$ and $u_{-}$ are non-negative, non-increasing and uniformly bounded by $\theta^2$, we derive
$$
0\leq T_{2,2,+} \leq \frac{C_1 \theta^2}{n} \quad \text{and} \quad 0 \leq T_{2,2,-} \leq \frac{C_1 \theta^2}{n}.
$$
With $c_\theta$ as in \eqref{eqn:c_theta}, this implies
\begin{equation}\label{eq:T3}
|T_2| \leq 2C_1 \left( c_\theta \left(\frac{R_\theta(a)}{n}+ \frac{2d_K(W_n,D_\theta)}{an} \right) + \frac{\theta^2}{n}\right).
\end{equation}

Next we bound $|T_3|$. Since $|f_1'(x)|=|\theta a^\theta \mathds{1}(x>a)/x^{\theta+1}| \le \theta/a$ for all $x \in \R_+ \setminus\{a\}$, we obtain 
\begin{align*}
	|T_{3,1}|& \le \frac{1}{n} \sum_{k=1}^{l-1} \E \left[\bigg|f'_1\left(W_n+\frac{k-1}{n}+\frac{U}{n}\right)\bigg|\left(\mathds{1}(W_n \le a)+\mathds{1}(W_n >a)\right) \right]\nonumber\\
	&\le \frac{\theta(l-1)}{an}\Prob{W_n \le a} + \frac{l-1}{n}\E \left[\frac{\theta a^\theta}{W_n^{\theta+1}}\mathds{1}(W_n >a)\right].
\end{align*}
Thus, Lemma~\ref{lem:mom} with $u=0$ yields
$$	
|T_{3,1}| \le \theta (l-1) \left[\frac{R_{\theta}(a)}{n} + \frac{d_K(W_n,D_\theta)}{an}+ \zeta(\theta+1)\left(\frac{R_\theta(a)}{n} + \frac{2d_K(W_n,D_\theta)}{an}\right)\right].
$$
Since $u_{+}$ and $u_{-}$ are non-negative and uniformly bounded by $\theta^2$ (see Theorem \ref{thm.bdkol}), we have
$$
|T_{3,2,+}-T_{3,2,-}| \leq \frac{(l-1)\theta^2}{n}
$$
so that
\begin{equation}\label{eq:T1}
	|T_{3}| \le (l-1)\left[c_\theta \left(\frac{R_\theta(a)}{n} + \frac{2d_K(W_n,D_\theta)}{an}\right) + \frac{\theta^2}{n} \right].
\end{equation}

Finally, we bound $T_4$. Arguing similarly as for $T_{3,1}$, using that $f_1'(x)=-\frac{\theta a^\theta}{x^{\theta+1}} \mathds{1}(x>a)$ for $x\in\mathbb{R}_+\setminus\{a\}$ in the first step and Lemma~\ref{lem:mom} with $u=k/n$ in the second, we obtain
\begin{align}\label{eq:T41}
	T_{4,1}
	& \le \frac{\tau}{n} \sum_{k=l}^n \frac{1}{k} \E \left[\frac{a^\theta}{(W_n + k/n)^{\theta+1}}\mathds{1}\left(W_n+\frac{k}{n}>a\right)\right]\nonumber\\
	& \le \frac{\tau}{n} \sum_{k=l}^n \frac{\zeta(\theta + 1)}{k} \left( R_\theta (\max\{a, k/n\}) + \frac{2d_K(W,D_\theta)}{\max\{a, k/n\}}\right).
\end{align}

First consider the case when $\theta \ge1$. Since $\sum_{k=l}^n \frac{1}{k} \le 1 + \log (n/l)$ for $n \ge l$, we derive
$$
T_{4,1} \le \tau \zeta(\theta+1) (1+\log (n/l)) \left[\frac{R_{\theta}(a)}{n} + \frac{2 d_K(W_n,D_\theta)}{an}\right].
$$
Using that $u_{+}$ and $u_{-}$ are non-negative and uniformly bounded by $\theta^2$ (see Theorem \ref{thm.bdkol}), we have
\begin{equation}\label{eq:T42}
T_{4,2,+}+T_{4,2,-} \leq 2\tau\theta \frac{1+\log (n/l)}{n}
\end{equation}
so that
\begin{equation}\label{eq:T4}
	T_4 \le \frac{\tau}{\theta} (1+\log (n/l)) \left[c_\theta \left(\frac{R_\theta(a)}{n} + \frac{2d_K(W_n,D_\theta)}{an}\right) + \frac{2\theta^2}{n} \right].
\end{equation}
From \eqref{eq:Tsp}, \eqref{eq:T2}, \eqref{eq:T3}, \eqref{eq:T1} and \eqref{eq:T4} it follows that
\begin{align*}
|\mathbf{P}\{W_n\leq a) - \mathbf{P}\{D_\theta\leq a\}| \leq & c_\theta \bigg(2C_1+l + \frac{\tau}{\theta} \big(1+\log (n/l)\big) \bigg) \left(\frac{R_\theta(a)}{n}+ \frac{2d_K(W_n,D_\theta)}{an}\right) \\
& \qquad + \bigg(2C_1+l + \frac{2\tau}{\theta}\big(1+ \log(n/l)\big)\bigg) \frac{\theta^2}{n}.
\end{align*}
Recalling that $M_n=4 c_\theta \big(2C_1+l+ \frac{\tau}{\theta}\big(1+ \log(n/l)\big)\big)/n$, for $a \in [M_n,\infty)\setminus\mathcal{R}_n$ we have
\begin{equation*}
 \bigg(2C_1+l + \frac{\tau}{\theta} \big(1+\log (n/l)\big) \bigg)  \frac{c_\theta}{an} \le \frac{1}{4}
\end{equation*}
whence
\begin{align*}
\sup_{a\geq M_n} |\mathbf{P}\{W_n\leq a) - \mathbf{P}\{D_\theta\leq a\}| & = \sup_{a\in [M_n,\infty) \setminus \mathcal{R}_n} |\mathbf{P}\{W_n\leq a) - \mathbf{P}\{D_\theta\leq a\}|\\
& \leq \bigg(2C_1+l + \frac{2\tau}{\theta} \big(1+\log (n/l)\big) \bigg) \frac{c_\theta K_\theta + \theta^2}{n} + \frac{1}{2} d_K(W_n,D_\theta).
\end{align*}
Now Lemma \ref{lem:lb} yields
\begin{equation*}
\begin{split}
d_K(W_n,D_\theta) & \le K_\theta M_n + \bigg(2C_1+l + \frac{2\tau}{\theta} \big(1+\log (n/l)\big) \bigg) \frac{c_\theta K_\theta + \theta^2}{n} + \frac{1}{2} d_K(W_n,D_\theta) \\
& \leq \bigg(2C_1+l + \frac{2\tau}{\theta} \big(1+\log (n/l)\big) \bigg) \frac{5c_\theta K_\theta + \theta^2}{n} + \frac{1}{2} d_K(W_n,D_\theta)
\end{split}
\end{equation*}
so that
\begin{equation*}
d_K(W_n,D_\theta) \le \frac{C_2 l}{n} + \frac{C_3 \tau \log(n/l)}{n}  
\end{equation*}
with 
$$
C_2= 2\bigg(2C_1+1+\frac{2\tau}{\theta}\bigg) (5c_\theta K_\theta+\theta^2) \quad \text{and} \quad C_3=\frac{4}{\theta} (5c_\theta K_\theta+\theta^2).
$$
This proves the desired upper bound for $\theta \ge 1$ and $l\ge k_0$.

Next, we consider the case when $\theta \in (0,1)$ and $l\ge k_0$. From \eqref{eq:T41}, we have
\begin{align*}
T_{4,1} & \leq \frac{\tau}{n} \sum_{k=l}^n \frac{\zeta(\theta+1)}{k} \bigg( \frac{k}{n} \bigg)^{\theta-1} + \frac{\tau}{n} \sum_{k=l}^n \frac{\zeta(\theta+1)}{k} \frac{2 d_K(W,D_\theta)}{\sqrt{a} \sqrt{k/n} } \\
& = \frac{\tau \zeta(\theta+1)}{n^\theta} \sum_{k=l}^n \frac{1}{k^{2-\theta}} + \tau \zeta(\theta+1) \sum_{k=l}^n \frac{1}{k^{3/2}} \frac{2 d_K(W,D_\theta)}{\sqrt{na}} \\
& \leq \frac{\tau \zeta(\theta+1)\zeta(2-\theta)}{n^\theta} + \tau \zeta(\theta+1) \zeta(3/2) \frac{2d_K(W,D_\theta)}{\sqrt{na}}.
\end{align*}
This and \eqref{eq:T42} yield
$$
T_4 \le \frac{\tau \zeta(\theta+1)\zeta(2-\theta)}{n^\theta} + \tau \zeta(\theta+1) \zeta(3/2) \frac{2d_K(W,D_\theta)}{\sqrt{na}} + 2\tau\theta\frac{1+\log n}{n}.
$$
Combining this with \eqref{eq:Tsp}, \eqref{eq:T2}, \eqref{eq:T3} and \eqref{eq:T1}, and recalling $c_\theta$ in \eqref{eqn:c_theta}, we obtain
\begin{align*}
|\mathbf{P}\{W_n\leq a) - \mathbf{P}\{D_\theta\leq a\}| & \leq (2C_1 + l) c_\theta \left(\frac{R_\theta(a)}{n}+ \frac{2d_K(W_n,D_\theta)}{an}\right)+ (2C_1 + l) \frac{\theta^2}{n} \\
& \quad + \frac{\tau \zeta(\theta+1)\zeta(2-\theta)}{n^\theta} + \tau \zeta(\theta+1) \zeta(3/2) \frac{2d_K(W,D_\theta)}{\sqrt{na}} + 2\tau\theta\frac{1+\log n}{n}.
\end{align*}
Since $a\geq M_n$ with
$$
M_n = \frac{\max\{6(2C_1+l)c_\theta, 36 \tau^2 \zeta(\theta+1)^2 \zeta(3/2)^2\}}{n},
$$
the coefficients of the Kolmogorov distances in the previous bound are at most $1/3$. Also we have that
\begin{equation}\label{eqn:bounds_M_n}
\frac{c_\theta l}{n} \leq M_n \leq \frac{c_{\theta,\tau} l}{n} \quad \text{with} \quad c_{\theta,\tau} = \max\{ 6(2C_1+1)c_\theta, 36 \tau^2 \zeta(\theta+1)^2 \zeta(3/2)^2 \}.
\end{equation}
Thus, noting that for $\theta\in (0,1)$,
$$
\frac{R_\theta(a)}{n}=\frac{a^{\theta-1}}{n} \leq \frac{M_n^{\theta-1}}{n} \leq \frac{c_\theta^{\theta-1} l^{\theta-1}}{n^\theta},
$$
we obtain
\begin{align*}
	\sup_{a\geq M_n} |\mathbf{P}\{W_n\leq a) - \mathbf{P}\{D_\theta\leq a\}| & = \sup_{a\in [M_n,\infty) \setminus \mathcal{R}_n} |\mathbf{P}\{W_n\leq a) - \mathbf{P}\{D_\theta\leq a\}|\\
	& \leq \frac{C_4 l^\theta}{n^\theta} + \frac{2}{3} d_K(W_n,D_\theta)
\end{align*}
with
$$
C_4 = (2C_1+1) c_\theta^\theta + (2C_1+1) \theta^2 + \tau \zeta(\theta+1)\zeta(2-\theta) + 2\tau\theta \left(1+ \sup_{y\in(1,\infty)}\frac{\log y}{y^{1-\theta}} \right).
$$
From Lemma \ref{lem:lb} and \eqref{eqn:bounds_M_n} it follows that
\begin{equation*}
	d_K(W_n,D_\theta) \le  \left(\frac{c_{\theta,\tau} l}{n}\right)^\theta + \frac{C_4 l^\theta}{n^\theta} + \frac{2}{3} d_K(W_n,D_\theta),
\end{equation*}
yielding
\begin{equation*}
	d_K(W_n,D_\theta) \le \frac{C_5 l^\theta}{n^\theta}
\end{equation*}
with $C_5= 3 (C_4+c_{\theta,\tau}^\theta)$. 
This shows the desired bound for $\theta \in (0,1)$ and $l\geq k_0$.

Finally, we study the case when $l < k_0$. We rewrite $W_n$ as
$$
W_n = \frac{1}{n} \sum_{k=l}^{k_0 -1} k Z_k + \frac{1}{n} \sum_{k=k_0}^n k Z_k =: R_n + W_n'.
$$
Since $R_n$ and $W_n'$ are independent, it follows from \eqref{eqn:d_K_approximation} that
$$
d_K(W_n,D_\theta) \leq d_K(W_n',D_\theta) + \tilde{r}_n
$$
with
$$
\tilde{r}_n:=\begin{cases} K_\theta \mathbf{E} R_n, & \quad \theta\geq 1,\\ (\mathbf{E} R_n)^{\theta}, & \quad \theta\in(0,1). \end{cases}
$$
For the Kolmogorov distance between $W_n'$ and $D_\theta$ we have already proven the assertion. Combining this with the observation that
$$
\mathbf{E}[R_n] = \frac{1}{n} \sum_{k=l}^{k_0-1} k \mathbf{E}[Z_k] \leq \frac{1}{n} \sum_{k=l}^{k_0-1} \left(\theta + \frac{\tau}{k}\right)\leq \frac{(\theta+\tau) k_0}{n}
$$
completes the proof.
\end{proof}

\begin{proof}[Proofs of upper bounds in Theorems~\ref{thm.ber} and \ref{thm:Poi}] We start by proving the upper bound in Theorem~\ref{thm.ber}. We have
	\begin{equation}\label{eq:tau}
		\left|\frac{\theta}{k+\beta} - \frac{\theta}{k}\right| \le \frac{\theta |\beta|}{k (k+\beta)} \le \frac{\theta |\beta|}{k^2} \frac{k}{k+\beta} \leq \frac{\theta |\beta|}{k^2} \bigg(1-\frac{\beta}{k+\beta}\bigg) \leq \frac{\theta |\beta|}{k^2} \bigg(1+\frac{|\beta|}{\theta}\bigg) = \frac{|\beta| (\theta+|\beta|)}{k^2}
	\end{equation}
	for $k\geq l \ge \theta-\beta$. From Theorem \ref{thm.ber'} with $m_k=1$, $p_k^{(1)}=\frac{\theta}{k+\beta}$ for $k\geq l$, and $\tau=|\beta|(\theta+|\beta|)$, it follows that
	$$
	d_K(W_n,D_\theta) \le \begin{dcases} C \left(\frac{l + |\beta| \log (n/l)}{n}\right), & \quad \theta\geq 1,\\  \frac{C l^\theta}{n^\theta}, & \quad \theta \in (0,1), \end{dcases}
	$$
	with a constant $C\in(0,\infty)$ depending only $\theta$ and $\beta$.
	
	Next, we prove the upper bound in Theorem~\ref{thm:Poi}. For $\frac{\theta}{l+\beta} \le m\in\N$, let $(Z_k)_{k\in\mathbb{N}}$ be independent random variables such that $Z_k$ follows a binomial distribution with parameters $m$ and $\frac{\theta}{m(k+\beta)}$ and let $W_{n,m}=\frac{1}{n}\sum_{k=l}^n k Z_k$. From Theorem \ref{thm.ber'} with $\tau=|\beta|(\theta+|\beta|)$ (see \eqref{eq:tau}), it follows that there exists a constant $C\in(0,\infty)$ depending only on $\theta$ and $\beta$ such that
	$$
	d_K(W_{n,m},D_\theta) \leq \begin{dcases} C \left(\frac{l + |\beta| \log(n/l)}{n}\right), & \quad \theta\geq 1,\\  \frac{C l^\theta}{n^\theta}, & \quad \theta \in (0,1). \end{dcases}
	$$
	Notice that
		$$
		d_K(W_n, D_\theta) \le d_K(W_{n,m},D_\theta) + \Prob{W_{n,m}\neq W_n}
		$$
		for any coupling of $W_n$ and $W_{n,m}$ on the same probability space. Since $W_n$ and $W_{n,m}$ can be coupled in such a way that $\Prob{W_{n,m}\neq W_n}\to0$ as $m\to\infty$ and $C$ does not depend on $m$, this proves the upper bound in Theorem~\ref{thm:Poi}. 
\end{proof}

\section{Lower bounds}\label{sec:lowerbound}
In this section, we develop some general tools culminating in Proposition~\ref{prop:optimality} providing a general lower bound on the Kolmogorov distances between $W_n$ as in Theorems~\ref{thm.ber} and \ref{thm:Poi} and $D_\theta$ for $\theta \ge 1$. We start with the following proposition that provides a lower bound on the Kolmogorov distance between a re-scaled integer-valued random variable and a Dickman distributed random variable. The proof uses the same ideas as in the proof of \cite[Lemma 4.1]{PRR13}.
\begin{prop}\label{prop:intapp}
	For each $\theta>0$ there exists a constant $c_\theta \in (0,\infty)$ such that for all random variables $S$ taking values in $\{0\} \cup (\mathbb{Z}\cap [l,\infty))$ for some $l\in\mathbb{N}$ and all $n\in\mathbb{N}$ with $n\geq l$,
	$$
	d_K(n^{-1}S, D_\theta) \ge \begin{cases} \frac{c_\theta l^\theta}{n^\theta}, & \quad \theta\in(0,1],\\ \frac{c_\theta}{n}, & \quad \theta>1. \end{cases} 
	$$
\end{prop}
\begin{proof} When $\theta \in (0,1]$, note that
	$$
	\Prob{n^{-1}S \in (0,l/n)} = 0 \quad \text{and} \quad \Prob{D_\theta \in (0,l/n)} = \frac{e^{-\theta \gamma} l^\theta}{\Gamma(\theta+1) n^\theta},
	$$
	where the second identity follows from \eqref{eqn:P_D_theta} and the continuity of the distribution of $D_\theta$.	Thus,
	\begin{align*}
		\frac{e^{-\theta \gamma} l^\theta}{\Gamma(\theta+1) n^\theta} &=\left| \Prob{n^{-1}S \in (0,l/n)} - \Prob{D_\theta \in(0,l/n)}\right| \\
		&=\lim_{ \eps \searrow 0} \left| \Prob{n^{-1}S \in (0,l/n-\eps]} - \Prob{D_\theta \in(0,l/n-\eps]}\right| \le 2 d_K(n^{-1}S,D_\theta)
	\end{align*}
	proving the assertion for $\theta \in (0,1]$.
	For $\theta >1$, letting $I_n = (1-(2n)^{-1}, 1)$, by \eqref{eqn:P_D_theta} and the mean value theorem, we derive
	\begin{align*}
		\Prob{D_\theta\in I_n}&=\frac{e^{-\theta \gamma}}{\Gamma(\theta+1)} \left[1 - \left(1-\frac{1}{2n}\right)^\theta \right] 
		\ge \frac{e^{-\theta \gamma}}{\Gamma(\theta+1)} \left[\frac{\theta}{2n}\left(1-\frac{1}{2n}\right)^{\theta-1} \right] \ge \frac{C_\theta}{n}
	\end{align*}
	for some constant $C_\theta>0$ depending only on $\theta$. Since $I_n$ does not contain multiples of $1/n$, but $n^{-1}S$ can only take multiples of $1/n$ as values, we have $\Prob{n^{-1}S \in I_n}=0$. Arguing similarly as in the case of $\theta \in (0,1]$, this yields
	$$
	\frac{C_\theta}{n} \le \left| \Prob{n^{-1}S \in I_n} - \Prob{D_\theta \in I_n}\right|  \le 2 d_K(W_n,D_\theta)
	$$
	proving the result for $\theta> 1$.
\end{proof}

For the sums of independent random variables considered in this paper, it turns out that the lower bound provided in Proposition~\ref{prop:intapp} is optimal only for $\theta \in (0,1)$. In the rest of the section, we thus focus on the case when $\theta \ge 1$ and provide an improved lower bound in Proposition~\ref{prop:optimality}. We first prove two important lemmas.

\begin{lemma}\label{lem:D_theta_one}
For each $\theta\geq 1$ there exists a constant $\underline{c}_\theta\in(0,\infty)$ such that for independent random variables $D_\theta$ and $\Delta$ with $\Delta\geq 0$ $\mathbf{P}$-a.s.\ and $\mathbf{E}[\Delta^2]<\infty$,
$$
\bigg| \Prob{D_\theta \leq 1} - \Prob{D_\theta \leq 1-\Delta} - \frac{e^{-\theta\gamma}}{\Gamma(\theta)} \mathbf{E}[\Delta] \bigg| \leq \underline{c}_\theta \mathbf{E}[\Delta^2].
$$
\end{lemma}

\begin{proof}
For $\delta\in[0,1/2]$, by the Taylor series approximation one has
$$
\Prob{D_\theta \leq 1} - \Prob{D_\theta \leq 1-\delta} = \frac{e^{-\theta\gamma}}{\Gamma(\theta)} \delta - \frac{(\theta-1) e^{-\theta\gamma} z^{\theta-2}}{2\Gamma(\theta)} \delta^2
$$
for some $z\in[1-\delta,1]$. Choosing $\delta=\Delta$ when $\Delta\leq \frac{1}{2}$ yields
\begin{align*}
\Prob{D_\theta \leq 1} - \Prob{D_\theta \leq 1-\Delta} & = \mathbf{E}\bigg[ \mathds{1}\{ \Delta\leq 1/2 \} \bigg( \frac{e^{-\theta\gamma}}{\Gamma(\theta)} \Delta - \frac{(\theta-1) e^{-\theta\gamma} Z^{\theta-2}}{2\Gamma(\theta)} \Delta^2 \bigg) \bigg] \\
& \quad + \Prob{D_\theta \leq 1, \Delta>1/2} - \Prob{D_\theta \leq 1-\Delta, \Delta>1/2}
\end{align*}
with a random variable $1-\Delta \le Z \le 1$ $\mathbf{P}$-a.s. Since $\mathds{1}\{ \Delta> 1/2 \} \le 2 \Delta$, we have
$$
\mathbf{E}[ \mathds{1}\{ \Delta> 1/2 \} \Delta ] \leq 2 \mathbf{E}[\Delta^2],
$$
while a simple use of Markov's inequality yields
$$
\Prob{D_\theta \leq 1, \Delta>1/2} - \Prob{D_\theta \leq 1-\Delta, \Delta>1/2} \leq \Prob{\Delta>1/2} \leq 4 \mathbf{E}[\Delta^2].
$$
Thus, we obtain
$$
\bigg| \Prob{D_\theta \leq 1} - \Prob{D_\theta \leq 1-\Delta} - \frac{e^{-\theta\gamma}}{\Gamma(\theta)} \mathbf{E}[\Delta] \bigg| \leq \bigg( \frac{2e^{-\theta\gamma}}{\Gamma(\theta)} + \frac{2(\theta-1) e^{-\theta\gamma}}{2\Gamma(\theta)} + 4 \bigg) \mathbf{E}[\Delta^2],
$$
which concludes the proof.
\end{proof}

\begin{lemma}\label{lem:lower_bound}
For $\theta \ge 1$, $\underline{c}_\theta$ as in Lemma~\ref{lem:D_theta_one}, and independent non-negative random variables $S$, $\Delta_1$ and $\Delta_2$ with $\mathbf{E}[\Delta_1^2],\mathbf{E}[\Delta_2^2]<\infty$,
$$
d_K(S+\Delta_2,D_\theta) \geq \frac{e^{-\theta\gamma}}{\Gamma(\theta)} |\mathbf{E}[\Delta_1-\Delta_2]| - d_K(S+\Delta_1,D_\theta) - \underline{c}_\theta \big(\mathbf{E}[ \Delta_1^2 ] + \mathbf{E}[\Delta_2^2] \big).
$$
\end{lemma}

\begin{proof}
Throughout this proof we assume that $D_\theta$ lives on the same probability space as $S$, $\Delta_1$ and $\Delta_2$ and is independent from them. First let $\mathbf{E}[\Delta_1-\Delta_2]\geq 0$. Then, we obtain
\begin{align*}
d_K(S+\Delta_2,D_\theta) &\geq \Prob{S+\Delta_2\leq 1-\Delta_1} - \Prob{ D_\theta\leq 1-\Delta_1} \\
& = \Prob{S+\Delta_1\leq 1-\Delta_2} - \Prob{D_\theta\leq 1-\Delta_2} \\
& \qquad  + \Prob{ D_\theta\leq 1}- \Prob{ D_\theta\leq 1-\Delta_1} + \Prob{D_\theta\leq 1-\Delta_2} - \Prob{ D_\theta\leq 1}\\
& \geq -d_K(S+\Delta_1,D_\theta) + \frac{e^{-\theta\gamma}}{\Gamma(\theta)} \mathbf{E}[\Delta_1-\Delta_2] - \underline{c}_\theta \big(\mathbf{E}[ \Delta_1^2 ] + \mathbf{E}[\Delta_2^2] \big),
\end{align*}
where we used Lemma \ref{lem:D_theta_one} in the last step. For  $\mathbf{E}[\Delta_1-\Delta_2] < 0$, we analogously have
\begin{align*}
d_K(S+\Delta_2,D_\theta) & \geq \Prob{ D_\theta\leq 1-\Delta_1} - \Prob{S+\Delta_2\leq 1-\Delta_1} \\
& = \Prob{D_\theta\leq 1-\Delta_2} - \Prob{S+\Delta_1\leq 1-\Delta_2}\\
& \qquad + \Prob{ D_\theta\leq 1} - \Prob{D_\theta\leq 1-\Delta_2} + \Prob{ D_\theta\leq 1-\Delta_1} -\Prob{ D_\theta\leq 1} \\
& \geq -d_K(S+\Delta_1,D_\theta) + \frac{e^{-\theta\gamma}}{\Gamma(\theta)} \mathbf{E}[\Delta_2-\Delta_1] - \underline{c}_\theta \big(\mathbf{E}[ \Delta_1^2 ] + \mathbf{E}[\Delta_2^2] \big).
\end{align*}
Combining the two cases completes the proof.
\end{proof}

For $\theta>0$ we write $\lceil \theta \rceil$ and $\lfloor \theta \rfloor$ to denote the smallest integer larger than or equal to $\theta$ and the largest integer smaller than or equal to $\theta$, respectively. We will repeatedly use the fact that for $\theta>0$, $\beta \in \R$ and $l,n \in \N$ with $ n \ge l > -\beta$, there exist constants $\tilde{c}_1,\tilde{c}_2\in(0,\infty)$ depending only on $\theta$ and $\beta$ such that
\begin{equation}\label{eqn:approximations_sums}
	\bigg| \frac{1}{n} \sum_{k=\lceil \theta \rceil}^{l-1} \theta - \frac{\theta l}{n} \bigg| \leq \frac{\tilde{c}_1}{n} \quad \text{and} \quad \bigg| \frac{1}{n} \sum_{k=l}^n \frac{\beta\theta}{k+\beta} - \frac{\beta\theta \log(n/l)}{n} \bigg| \leq \frac{\tilde{c}_2}{n},
\end{equation}
where throughout the sequel, we take as a convention $\sum_{k=\lceil \theta \rceil}^{l-1}$ to be zero when $l <1+\lceil \theta \rceil$.
\begin{prop}\label{prop:optimality}
Let $\theta\geq 1$ and $\beta\in\mathbb{R}$.
\begin{enumerate}[(a)]
\item [(a)] Let $W_n$ be as in Theorem \ref{thm:Poi} with $l\in\mathbb{N}$ such that $l>-\beta$. There exists a constant $\underline{c}_{\theta,\beta}\in(0,\infty)$ depending only on $\theta$ and $\beta$ such that
$$
d_K(W_n,D_\theta) \geq \frac{\theta e^{-\theta\gamma}}{\Gamma(\theta)} \bigg| \frac{l}{n} + \beta \frac{\log(n/l)}{n} \bigg| - \underline{c}_{\theta,\beta} \bigg( \frac{l^2}{n^2} + \frac{1}{n} \bigg)
$$
for all $n\in\mathbb{N}$ with $n \ge l$.
\item [(b)] Let $W_n$ be as in Theorem~\ref{thm.ber} with $l\in\mathbb{N}$ such that $l\geq\theta-\beta$. If $\beta\geq 0$, there exists a constant $\underline{c}_{\theta,\beta}\in(0,\infty)$ depending only on $\theta$ and $\beta$ such that
\begin{equation}\label{eqn:lower_bound_Bernoulli_I}
d_K(W_n,D_\theta) \geq \frac{\theta e^{-\theta\gamma}}{\Gamma(\theta)} \bigg| \frac{l}{n} + \beta \frac{\log(n/l)}{n} \bigg| - \underline{c}_{\theta,\beta} \bigg( \frac{l^2}{n^2} + \frac{1}{n} \bigg)
\end{equation}
for all $n\in\mathbb{N}$ with $n \ge l$. If $\beta<0$, there exists a constant $\underline{c}_{\theta,\beta} \in(0,\infty)$ depending only on $\theta$ and $\beta$ such that for $n\in\mathbb{N}$,
\begin{equation}\label{eqn:lower_bound_Bernoulli_II}
d_K(W_n,D_\theta) \geq \frac{\theta e^{-\theta\gamma}}{\Gamma(\theta)}  \bigg|\frac{|\beta|\log (n/l)}{n} - \frac{l}{n} \bigg| - \frac{\theta e^{-\theta\gamma}}{\Gamma(\theta)} \frac{|\beta|\log (n/l)}{n}\sum_{k=l}^n \frac{\theta|\beta|}{(k+\beta)(k-\theta)} - \underline{c}_{\theta,\beta} \bigg( \frac{l^2}{n^2} + \frac{1}{n}\bigg).
\end{equation}
\end{enumerate}
\end{prop}

\begin{remark}\label{rem:opt}
	Note that the lower bounds in Theorems~\ref{thm.ber} and \ref{thm:Poi} are for fixed $l$  
	and $n\to\infty$. A more general situation is when $l$ is allowed to  
	depend on $n$. For this case we obtain lower bounds of the same orders  
	as the upper bounds in Theorems~\ref{thm.ber} and \ref{thm:Poi} for $\theta\in(0,1)$ from  
	Proposition~\ref{prop:intapp}, and for $\theta\ge 1$ and $\beta> 0$, from  
	Proposition~\ref{prop:optimality}. In the case when $\theta \ge 1$ and $\beta=0$, the optimal order follows from Propositions~\ref{prop:intapp} and \ref{prop:optimality}, depending on the size of $l$. For $\beta<0$ and $\theta\ge 1$ it can happen that  
	the two terms $l/n$ and $\beta \log (n/l)/n$ inside the absolute values on the right-hand sides of the  
	lower bounds in Proposition~\ref{prop:optimality} cancel each other out so that the lower bounds are not of the desired orders. This behaviour for $\beta<0$ leads to  
	the following open question. If we think of $\mathbf{P}(W_n\leq t)$ for  
	fixed $n\in\mathbb{N}$ and $t\geq 0$ as a function of $l$ and $\beta$,  
	it is non-decreasing in both arguments. Since for $\beta=0$ and fixed $l \in \N$, by Theorems~\ref{thm.ber} and \ref{thm:Poi} we obtain a faster $1/n$ rate of convergence, one can  
	now wonder if, for $\beta<0$, it is possible to choose $l$ as a function of $n$ in such a way that these opposite  
	effects cancel out uniformly over all $t\geq 0$ and one still obtains a  
	$1/n$ rate of convergence.
\end{remark}

\begin{proof}[Proof of Proposition~\ref{prop:optimality}]

First we consider the Poisson case (a). For $\beta\geq 0$, define $\Delta_{1,n}$ independent of $W_n= n^{-1} \sum_{k=l}^n kP_k$ as
$$
\Delta_{1,n}=\frac{1}{n}\sum_{k=\lceil \theta \rceil}^{l-1} k P_k' + \frac{1}{n} \sum_{k=l}^n k P_k''
$$
with independent random variables $P_k'\sim \operatorname{Poi}\big( \frac{\theta}{k} \big)$ for $k\in\{\lceil \theta \rceil,\hdots,l-1\}$ and $P_k''\sim \operatorname{Poi}\big( \frac{\theta}{k} - \frac{\theta}{k+\beta} \big)$ for $k\in\{l,\hdots,n\}$. By \eqref{eqn:approximations_sums}, we have
$$
\mathbf{E}[\Delta_{1,n}] = \frac{1}{n} \sum_{k=\lceil \theta \rceil}^{l-1} \theta + \frac{1}{n} \sum_{k=l}^n \frac{\beta \theta}{k+\beta} \geq \frac{\theta l}{n} + \frac{\beta \theta \log(n/l)}{n} - \frac{\tilde{c}_1+\tilde{c}_2}{n},
$$
while
$$
\mathbf{E}[\Delta_{1,n}^2] = \mathbf{E}[\Delta_{1,n}]^2 + \frac{1}{n^2} \sum_{k=\lceil \theta \rceil}^{l-1} \theta k + \frac{1}{n^2} \sum_{k=l}^n \frac{\beta\theta k}{k+\beta} \leq \mathbf{E}[\Delta_{1,n}]^2 + \frac{\theta l^2}{n^2} + \frac{\beta\theta}{n}.
$$
From the upper bound in Theorem \ref{thm:Poi}, we know that
$$
d_K(W_n+\Delta_{1,n},D_\theta) \leq \frac{C_\theta}{n}
$$
for some $C_\theta \in (0,\infty)$ depending only on $\theta$. Now the assertion follows from Lemma \ref{lem:lower_bound} with $S=W_n$, $\Delta=\Delta_{1,n}$ and $\Delta_2=0$.

For the case when $\beta <0$, we choose
$$
S_n = \frac{1}{n} \sum_{k=l}^{n} k P_k', \quad \Delta_{1,n} = \frac{1}{n} \sum_{k=\lceil \theta \rceil}^{l-1} k P_k'' \quad \text{and} \quad \Delta_{2,n} = \frac{1}{n} \sum_{k=l}^n k P_k'''
$$
with independent $P_k'\sim \operatorname{Poi}\big( \frac{\theta}{k} \big)$ and $P_k'''\sim \operatorname{Poi}\big( \frac{\theta}{k+\beta} - \frac{\theta}{k} \big)$ for $k\in\{l,\hdots,n\}$, and $P_k''\sim \operatorname{Poi}\big( \frac{\theta}{k} \big)$ for $k\in\{\lceil \theta \rceil,\hdots,l-1\}$. Note that
$$
\mathbf{E}[\Delta_{1,n}-\Delta_{2,n}] = \frac{1}{n} \sum_{k=\lceil \theta \rceil}^{l-1} \theta + \frac{1}{n} \sum_{k=l}^n \frac{\beta \theta}{k+\beta}
$$
so that, by \eqref{eqn:approximations_sums},
$$
|\mathbf{E}[\Delta_{1,n}-\Delta_{2,n}]| \ge \theta \bigg| \frac{l}{n} + \beta \frac{\log(n/l)}{n} \bigg| - \frac{\tilde{c}_1+\tilde{c}_2}{n}.
$$
Moreover, we have
$$
\mathbf{E}[\Delta_{1,n}^2] = \mathbf{E}[\Delta_{1,n}]^2 + \frac{1}{n^2} \sum_{k=\lceil \theta \rceil}^{l-1} \theta k \quad \text{and} \quad \mathbf{E}[\Delta_{2,n}^2] = \mathbf{E}[\Delta_{2,n}]^2 - \frac{1}{n^2} \sum_{k=l}^n \frac{\beta \theta k}{k+\beta} \leq \mathbf{E}[\Delta_{2,n}]^2 - \frac{\beta \theta l}{(l+\beta) n}.
$$
Now the upper bound in Theorem \ref{thm:Poi} yields
$$
d_K(S_n+\Delta_{1,n},D_\theta) \leq \frac{C_\theta}{n}.
$$
Note that for $\N \ni l >-\beta$, one has $l/(l+\beta) \le c_\beta$ for some $c_\beta \in (0,\infty)$ depending only on $\beta$. Since $W_n$ has the same distribution as $S_n+\Delta_{2,n}$, an application of Lemma \ref{lem:lower_bound} now completes the proof of (a).

We move on to the proof of the Bernoulli case (b). For $\beta\geq 0$, define $\Delta_{1,n}$ independent of $W_n$ as
$$
\Delta_{1,n}= \frac{1}{n} \sum_{k=\lceil \theta\rceil}^{l-1} k B_k' + \frac{1}{n} \sum_{k=l}^n k B_k''
$$
with independent random variables $B_k'\sim \operatorname{Ber}\big( \frac{\theta}{k} \big)$ for $k\in\{\lceil \theta\rceil,\hdots,l-1\}$ and $B_k''\sim \operatorname{Ber}\big( \frac{\theta}{k}-\frac{\theta}{k+\beta} \big)$ for $k\in\{l,\hdots,n\}$. Together with \eqref{eqn:approximations_sums}, we obtain
$$
\mathbf{E}[\Delta_{1,n}] = \frac{1}{n} \sum_{k=\lceil \theta\rceil}^{l-1} \theta + \frac{1}{n} \sum_{k=l}^n \frac{\beta\theta}{k+\beta} \geq \frac{\theta l}{n} + \frac{\beta \theta \log(n/l)}{n} - \frac{\tilde{c}_1+\tilde{c}_2}{n}
$$
and
$$
\mathbf{E}[\Delta_{1,n}^2] \leq \mathbf{E}[\Delta_{1,n}]^2 + \frac{1}{n^2} \sum_{k=\lceil \theta\rceil}^{l-1}\theta k + \frac{1}{n^2} \sum_{k=l}^n \frac{\beta\theta k}{k+\beta}.
$$
Using Theorem \ref{thm.ber'} with $\tau=0$, we have that
\begin{equation*}
d_K(W_n+\Delta_{1,n},D_\theta) \le \frac{C_{\theta}}{n}
\end{equation*}
with a constant $C_{\theta}$ depending only on $\theta$. Now Lemma \ref{lem:lower_bound} with $S=W_n$, $\Delta_1=\Delta_{1,n}$ and $\Delta_2=0$ leads to \eqref{eqn:lower_bound_Bernoulli_I}.

Finally, we consider the case when $\beta<0$. Let $(B_{k}')_{k\geq \lceil \theta\rceil}$ be i.i.d.\ random variables such that $B_k'\sim\operatorname{Ber}(\theta/k)$ and define
$$
W_n'=\frac{1}{n} \sum_{k=l}^n k B_k' \quad \text{and} \quad \Delta_n = \frac{1}{n} \sum_{k=\lceil\theta\rceil}^{l-1} k B_k'.
$$
Moreover, let $(B_{k}'')_{k\geq \lceil \theta - \beta\rceil}$ be i.i.d.\ random variables independent from $(B_{k}')_{k\geq \lceil \theta\rceil}$ and satisfying $B_k''\sim\operatorname{Ber}\bigg(\frac{\theta |\beta|}{(k+\beta)(k-\theta)}\bigg)$. Notice that $W_n$ has the same distribution as
$$
\frac{1}{n} \sum_{k=l}^n k (B_k' + \mathds{1}\{B_k'=0\} B_k'' ),
$$
and in the sequel we define $W_n$ on the same probability space in this way. For $D_\theta$ independent of $\Delta_n$, by the independence of $W_n$ and $\Delta_n$ we have
\begin{align}\label{eqn:lower_bound_coupling}
d_K(W_n,D_\theta) & \geq \Prob{ D_\theta \leq 1 - \Delta_n } - \Prob{ W_n \leq 1 - \Delta_n } \nonumber\\
& \geq \Prob{ W_n'+\Delta_n \leq 1} - \Prob{ W_n + \Delta_n \leq 1 } + \Prob{D_\theta\leq 1} - \Prob{ W_n'+\Delta_n \leq 1} \nonumber\\
& \qquad + \Prob{D_\theta\leq 1- \Delta_n } - \Prob{D_\theta\leq 1}.
\end{align}
In the sequel we use the abbreviations $S_n=W_n+\Delta_n$ and $S_n'=W_n'+\Delta_n$. For $k\geq l$ let $A_k$ be the event that $B_k'=0$, $B''_k=1$ and $B''_j=0$ for all $j\neq k$ and $l \le j \le n$. Note that with $p_{l,n}:=\Prob{ B''_j=0 \text{ for all } l \le j \le n }$,
$$
\Prob{A_k}\geq \frac{p_{l,n} \theta |\beta|}{(k+\beta)k} 
$$
for $l \le k \le n$. For $k\in\mathbb{N}$ with $l\leq k\leq n$ let $S_{n,k}'= S_n'- \frac{k B_k'}{n}$. Notice that, by our coupling, $W_n$ dominates $W_n'$ and if $A_k$ occurs for some $l\leq k\leq n$, then $W_n = W_n' + k/n$. Thus,
\begin{align}
\Prob{ W_n'+\Delta_n \leq 1} - \Prob{ W_n + \Delta_n \leq 1 }
& = \Prob{ S_n' \leq 1 } - \Prob{ S_n \leq 1 }\nonumber\\ 
& \geq \sum_{k=l}^n \Prob{ S_n' \leq 1, A_k} - \Prob{ S_n \leq 1, A_k } \nonumber \\
& = \sum_{k=l}^n \Prob{ S_{n,k}' \leq 1, A_k } - \mathbf{P}\bigg\{ S_{n,k}' + \frac{k}{n} \leq 1, A_k \bigg\} \nonumber \\
& \geq \sum_{k=l}^n \frac{p_{l,n} \theta |\beta|}{k (k+\beta)} \bigg(\Prob{ S_{n,k}' \leq 1} - \mathbf{P}\bigg\{ S_{n,k}' + \frac{k}{n} \leq 1\bigg\}\bigg), \label{eqn:lower_bound_W_W'}
\end{align}
where in the final step, we have used the independence of $S_{n,k}'$ and $A_k$. Conditioning on $B_k'$ we have
$$
\Prob{ S_n' \leq 1 } = \frac{k-\theta}{k} \Prob{ S_{n,k}' \leq 1 } + \frac{\theta}{k} \mathbf{P}\bigg\{ S_{n,k}' + \frac{k}{n} \leq 1 \bigg\}
$$
so that
\begin{align*}
\frac{k-\theta}{k} \bigg( \Prob{ S_{n,k}' \leq 1 } - \mathbf{P}\bigg\{ S_{n,k}' + \frac{k}{n} \leq 1 \bigg\} \bigg) = \Prob{ S_n' \leq 1 } - \mathbf{P}\bigg\{ S_{n,k}' + \frac{k}{n} \leq 1 \bigg\}.
\end{align*}
Since
$$
\mathbf{P}\bigg\{ S_n' + \frac{k}{n} \leq 1 \bigg\} = \frac{k-\theta}{k} \mathbf{P}\bigg\{ S_{n,k}' + \frac{k}{n} \leq 1 \bigg\} + \frac{\theta}{k} \mathbf{P}\bigg\{ S_{n,k}' + \frac{2k}{n} \leq 1 \bigg\}
$$
and, thus,
\begin{align*}
\mathbf{P}\bigg\{ S_{n,k}' + \frac{k}{n} \leq 1 \bigg\} & =  \mathbf{P}\bigg\{ S_{n}' + \frac{k}{n} \leq 1 \bigg\} + \frac{\theta}{k-\theta} \bigg( \mathbf{P}\bigg\{ S_{n}' + \frac{k}{n} \leq 1 \bigg\} - \mathbf{P}\bigg\{ S_{n,k}' + \frac{2k}{n} \leq 1 \bigg\}  \bigg),
\end{align*}
plugging in above, we obtain
\begin{align}\label{eq:Sbd}
&\frac{k-\theta}{k}\bigg(\Prob{S'_{n,k} \leq 1 } - \mathbf{P}\bigg\{S'_{n,k} + \frac{k}{n} \leq 1 \bigg\} \bigg) \nonumber\\
&= \Prob{S'_{n} \leq 1 } - \mathbf{P}\bigg\{S'_{n} + \frac{k}{n} \leq 1 \bigg\} 
- \frac{\theta}{k-\theta} \bigg( \mathbf{P}\bigg\{S'_{n} + \frac{k}{n} \leq 1 \bigg\} - \mathbf{P}\bigg\{S'_{n,k} + \frac{2k}{n} \leq 1 \bigg\}\bigg).
\end{align}
Note that by Lemma \ref{lem:D_theta_one}
\begin{align*}
\Prob{S'_{n} \leq 1 } - \mathbf{P}\bigg\{S'_{n} + \frac{k}{n} \leq 1 \bigg\}
& \geq \Prob{ D_\theta \leq 1 } - \mathbf{P}\bigg\{ D_\theta+\frac{k}{n} \leq 1 \bigg\} - 2 d_K(S_n',D_\theta) \\
& \geq \frac{e^{-\theta\gamma}}{\Gamma(\theta)} \frac{k}{n} - \underline c_\theta \frac{k^2}{n^2} - 2 d_K(S_n',D_\theta)
\end{align*}
and, by Lemma \ref{lem:bound_Dickman}
\begin{multline*}
0  \leq \mathbf{P}\bigg\{S'_{n} + \frac{k}{n} \leq 1 \bigg\} - \mathbf{P}\bigg\{S'_{n,k} + \frac{2k}{n} \leq 1 \bigg\} \\
 \leq \mathbf{P}\bigg\{S'_{n} + \frac{k}{n} \leq 1 \bigg\} - \mathbf{P}\bigg\{S'_{n} + \frac{2k}{n} \leq 1 \bigg\}  = \mathbf{P}\bigg\{ 1- \frac{2k}{n} < S'_{n} \leq 1 - \frac{k}{n} \bigg\} \\
 \leq \mathbf{P}\bigg\{ 1- \frac{2k}{n} < D_\theta \leq 1 - \frac{k}{n} \bigg\} + 2 d_K(S_n',D_\theta) \leq \frac{K_\theta k}{n} + 2 d_K(S_n',D_\theta).
\end{multline*}
Thus, plugging the above bounds in \eqref{eq:Sbd}, for $l \le k \le n$ we obtain
\begin{align*}
\Prob{S'_{n,k} \leq 1 } - \mathbf{P}\bigg\{S'_{n,k} + \frac{k}{n} \leq 1 \bigg\} & \geq \frac{k}{k-\theta} \bigg( \frac{e^{-\theta\gamma}}{\Gamma(\theta)} \frac{k}{n} - \underline c_\theta \frac{k^2}{n^2} - 2 d_K(S_n',D_\theta) \bigg) \\
& \quad - \frac{\theta k}{(k-\theta)^2} \bigg( \frac{K_\theta k}{n} + 2 d_K(S_n',D_\theta) \bigg).
\end{align*}
Combining this with \eqref{eqn:lower_bound_W_W'} leads to
$$
 \Prob{ W_n'+\Delta_n \leq 1} - \Prob{ W_n + \Delta_n \leq 1 } \geq \frac{e^{-\theta\gamma}}{\Gamma(\theta)} p_{l,n} \theta |\beta| \frac{1}{n} \sum_{k=l}^{n} \frac{k}{(k+\beta)(k-\theta)} - \frac{\underline{c}_1}{n} - \underline{c}_2 d_K(S_n',D_\theta)
$$
with constants $\underline{c}_1,\underline{c}_2\in(0,\infty)$ depending only on $\theta$ and $\beta$. From Lemma \ref{lem:D_theta_one} we deduce that
\begin{equation}\label{eqn:comparison_Delta_n}
\bigg|\Prob{D_\theta\leq 1- \Delta_n } - \Prob{D_\theta\leq 1} + \frac{e^{-\theta\gamma}}{\Gamma(\theta)} \frac{1}{n} \sum_{k=\lceil \theta \rceil}^{l-1} \theta \bigg| \leq \underline{c}_\theta \bigg( \frac{1}{n} \sum_{k=\lceil \theta \rceil}^{l-1} \theta \bigg)^2 + \underline{c}_\theta \frac{1}{n^2} \sum_{k=\lceil \theta \rceil}^{l-1} \theta k \leq \underline{c}_3 \frac{l^2}{n^2}
\end{equation}
for some constant $\underline{c}_3 \in (0,1)$ depending only on $\theta$. On the other hand,
$$
\Prob{D_\theta\leq 1} - \Prob{ W_n'+\Delta_n \leq 1} \ge -d_K(S_n',D_\theta).
$$
Now, the upper bound in Theorem \ref{thm.ber} implies
\begin{equation}\label{eqn:approximation_S_n'}
d_K(S_n',D_\theta) \leq \frac{C_{\theta}}{n}
\end{equation}
with a constant $C_{\theta} \in(0,\infty)$ depending only on $\theta$. Together with \eqref{eqn:lower_bound_coupling}, and
\begin{equation}\label{eq:plb}
p_{l,n} \geq 1 - \sum_{k=l}^n \frac{\theta|\beta|}{(k+\beta)(k-\theta)},
\end{equation}
the previous estimates lead to
\begin{equation}\label{eqn:lower_bound_Bernoulli_part_I}
d_K(W_n,D_\theta) \geq \frac{\theta e^{-\theta\gamma}}{\Gamma(\theta)} \bigg( \frac{|\beta|\log (n/l)}{n} - \frac{l}{n} \bigg) - \frac{\theta e^{-\theta\gamma}}{\Gamma(\theta)} \frac{|\beta|\log (n/l)}{n} \sum_{k=l}^n \frac{\theta|\beta|}{(k+\beta)(k-\theta)} - C \bigg(\frac{1}{n} + \frac{l^2}{n^2}\bigg)
\end{equation}
with a constant $C\in(0,\infty)$ depending only on $\theta$ and $\beta$.

For $k\in\mathbb{N}$ with $k\geq l \ge \theta - \beta$ let $p_k$ be such that
$$
(1-p_k)\frac{\theta}{k} +  p_k = \frac{\theta}{k+\beta},
$$
which is equivalent to
$$
p_k = \frac{|\beta| \theta}{(k-\theta)(k+\beta)}.
$$
Notice that $p_k$ is necessarily in $\in(0,1]$. Let $B_k'''\sim \operatorname{Ber}(p_k)$ be independent random variables that are also independent from $(B_k')_{k\geq \lceil \theta\rceil}$. Define
$$
W_n''=\frac{1}{n} \sum_{k=l}^n k (B_k'+B_k''').
$$
Since
$
\mathbf{P}\{B_k'+B_k'''\geq 1\}=(\theta/k)(1-p_k) + p_k = \frac{\theta}{k+\beta}$, we have $\mathbf{P}(W_n\leq t) \geq \mathbf{P}(W_n''\leq t)$ for all $t\in\mathbb{R}$. With $\Delta_n = \frac{1}{n} \sum_{k=\lceil\theta\rceil}^{l-1} k B_k'$ as above independent of $W_n$, this implies
\begin{equation}\label{eqn:upper_bound_coupling}
\begin{split}
d_K(W_n,D_\theta) & \geq \Prob{ W_n \leq 1 - \Delta_n } - \Prob{ D_\theta \leq 1 - \Delta_n } \\
& \geq \Prob{ W_n'' \leq 1 - \Delta_n } - \Prob{ D_\theta \leq 1 - \Delta_n } \\
& = \Prob{ W_n'' + \Delta_n \leq 1 } - \Prob{ D_\theta \leq 1 } + \Prob{ D_\theta \leq 1 } - \Prob{ D_\theta \leq 1 - \Delta_n }.
\end{split}
\end{equation}
Note that $W_n'' + \Delta_n=\frac{1}{n} \sum_{k=\lceil \theta \rceil}^n k B_k' + \frac{1}{n} \sum_{k=l}^n k B_k'''=: S_n'+R_n$. We have
$$
\mathbf{E}[R_n] = \frac{1}{n} \sum_{k=l}^n kp_k = \frac{1}{n} \sum_{k=l}^n \frac{|\beta| \theta k}{(k-\theta)(k+\beta)}
\quad \text{and} \quad \mathbf{E}[R_n^2 ] \leq \mathbf{E}[R_n]^2 + \frac{1}{n^2} \sum_{k=l}^n k^2 p_k.
$$
Since there exist constants $C_1,C_2,C_3\in(0,\infty)$ depending only on $\theta$ and $\beta$ such that
$$
\bigg| \mathbf{E}[R_n] - \frac{|\beta|\theta \log(n/l)}{n} \bigg| \leq \frac{C_1}{n} \quad \text{and} \quad
\mathbf{E}[ R_n^2] \leq C_2 \bigg( \frac{\log(n/l)^2}{n^2}  + \frac{1}{n}\bigg) \le \frac{C_3}{n},
$$
Lemma \ref{lem:D_theta_one} and \eqref{eqn:approximation_S_n'} lead to
\begin{multline*}
\Prob{ W_n'' + \Delta_n \leq 1 } - \Prob{ D_\theta \leq 1 }  = \Prob{ S_n' \leq 1 - R_n } -  \Prob{ D_\theta \leq 1 - R_n } \\
  \qquad \qquad \qquad \qquad - \left(\Prob{ D_\theta \leq 1 } - \Prob{ D_\theta \leq 1 - R_n } \right)\\
 \ge - d_K(S_n',D_\theta) - \frac{\theta e^{-\theta\gamma}}{\Gamma(\theta)} \frac{|\beta| \log(n/l) + C_1}{n} - \underline{c}_\theta  \frac{C_3}{n} \ge - \frac{\theta e^{-\theta\gamma}}{\Gamma(\theta)} \frac{|\beta| \log(n/l)}{n} - \frac{C_4}{n}
\end{multline*}
with a constant $C_4\in(0,\infty)$ depending only on $\theta$ and $\beta$. Thus, we obtain from \eqref{eqn:upper_bound_coupling} and \eqref{eqn:comparison_Delta_n} that
\begin{equation}\label{eqn:lower_bound_Bernoulli_part_II}
d_K(W_n,D_\theta) \geq \frac{\theta e^{-\theta\gamma}}{\Gamma(\theta)} \bigg( \frac{l}{n} - \frac{|\beta| \log(n/l)}{n} \bigg) -  \frac{\widetilde{C}_4}{n}- \underline{c}_3 \frac{l^2}{n^2}
\end{equation}
with a constant $\widetilde{C}_4\in(0,\infty)$ depending only on $\theta$ and $\beta$. Combining \eqref{eqn:lower_bound_Bernoulli_part_I} and \eqref{eqn:lower_bound_Bernoulli_part_II} proves \eqref{eqn:lower_bound_Bernoulli_II}.
\end{proof}

	Let $S$ and $\Delta$ be non-negative random variables with $\Delta \le a \in (0,\infty)$ almost surely, and let $\theta \ge 1$. Then, one has that
		\begin{equation}\label{eq:indadd'} 
		d_K(S+\Delta,D_\theta) \ge  d_K(S,D_\theta) - K_\theta a.
	\end{equation}
	Indeed, this follows by noticing first that for $t \in \R_+$ such that $\Prob{S \le t} \ge \Prob{D_\theta \le t}$, by Lemma~\ref{lem:bound_Dickman} we have
	\begin{align*}
		0 &\le \Prob{S \le t} - \Prob{D_\theta \le t} \\
		&\le \Prob{S +\Delta \le t+a} - \Prob{D_\theta \le t+a} + \Prob{D_\theta \le t+a}- \Prob{D_\theta \le t} \le d_K(S+\Delta, D_\theta) + K_\theta a.
	\end{align*}
On the other hand, for $t \in \R_+$ such that $\Prob{S \le t} < \Prob{D_\theta \le t}$, by Lemma~\ref{lem:bound_Dickman} we obtain
	$$
	0 < \Prob{D_\theta \le t} - \Prob{S \le t} \le \Prob{D_\theta \le t} - \Prob{S+\Delta \le t} \le d_K(S+\Delta, D_\theta).
	$$
Combining these two inequalities yields \eqref{eq:indadd'}.

\begin{proof}[Proofs of optimality of the bounds in Theorems~\ref{thm.ber} and \ref{thm:Poi}]
	To prove the optimality (in $n$) of the bound in Theorem~\ref{thm.ber}, first notice that when $\theta \in (0,1)$, or when $\theta \ge 1$ and $\beta=0$, the optimality of our bound follows directly from Proposition~\ref{prop:intapp}. In the case when $\theta\geq 1$ and $\beta > 0$, for $n$ large enough the desired lower bound holds due to \eqref{eqn:lower_bound_Bernoulli_I} in Proposition \ref{prop:optimality}(b). When $\theta \ge 1$ and $\beta<0$, notice that there exists a positive integer $N_{\theta,\beta}$ depending only on $\theta$ and $\beta$ such that for $m \ge N_{\theta,\beta}$,
		\begin{equation*}
			1 - \sum_{k=m}^n \frac{\theta|\beta|}{(k+\beta)(k-\theta)} \ge \frac{1}{2}.
		\end{equation*}
		Thus, for $n$ large enough, the result follows by \eqref{eqn:lower_bound_Bernoulli_II} in Proposition \ref{prop:optimality}(b), when $l \ge \max\{N_{\theta,\beta}, \theta-\beta\} =: \widetilde N_{\theta,\beta}$. If $\theta - \beta \le l <\widetilde N_{\theta,\beta}$, then noting that $\sum_{k=l}^{\widetilde N_{\theta,\beta} -1} k B_k \le \widetilde N^2_{\theta,\beta}$, by \eqref{eq:indadd'} and \eqref{eqn:lower_bound_Bernoulli_II} we have
		\begin{align*}
			d_K(W_n,D_\theta) &\ge d_K\left(\frac{1}{n} \sum_{k=\widetilde N_{\theta,\beta}}^n k B_k, D_\theta\right) - \frac{K_\theta {\widetilde N_{\theta,\beta}}^2}{n}\\
			& \ge \frac{\theta e^{-\theta\gamma}}{\Gamma(\theta)} \bigg( \frac{|\beta| \log (n/\widetilde N_{\theta,\beta})}{2n} - \frac{\widetilde N_{\theta,\beta}}{n} \bigg) - \underline{c}_{\theta,\beta} \bigg(\frac{1}{n} + \frac{{\widetilde N_{\theta,\beta}}^2}{n^2}\bigg) - \frac{K_\theta {\widetilde N_{\theta,\beta}}^2}{n}
		\end{align*}
		proving the desired lower bound for $n$ sufficiently large. Finally, notice that for small values of $n$, since $W_n$ is different from $D_\theta$ in distribution, we have $d_K(W_n,D_\theta) \ge c_n$ for some $c_n>0$. Thus, taking the minimum of these $c_n$'s as a lower bound for small values of $n$, the result follows for all $n \ge l$.
	
	Finally the optimality (in $n$) of the bound in Theorem~\ref{thm:Poi} follows from Proposition~\ref{prop:intapp} when $\theta \in (0,1)$, or when $\theta \ge 1$ and $\beta=0$, while Proposition \ref{prop:optimality}(a) yields the lower bound when $\theta \ge 1$ and $\beta \neq 0$ for $n$ large enough. Taking care of the small values of $n$ similarly as above, we obtain the desired result.
\end{proof}

\section{Proofs of Corollaries~\ref{cor:Q}, \ref{cor:IT} and Theorem~\ref{thm:imp}}\label{sec:mainresults} In this final section, we collect the proofs of our remaining results.

\begin{proof}[Proof of Corollary~\ref{cor:Q}]
	We start by showing (for completeness) a distributional equality that also follows from \cite[Equation~(3.2) and the two equations above it]{Hwa}. Without loss of generality, we assume that the numbers are $\{1,\hdots, n\}$, and we can generate the random keys by taking a uniform random permutation $\sigma$ of $\{1,\hdots,n\}$. Notice that a $k\in\{1,\hdots,n\}$ is used as a partitioning key if and only if $k$ appears before $1,\hdots, k-1$ in $\sigma$. In this case, we write $\sigma \in A_k$. It is straightforward to see that $\Prob{A_k} = 1/k$. Let $\xi(i,j)$, $1 \le i,j \le n$, be the indicator that the key $i$, as a partitioning key, compares with the key $j$ at some point of time in the whole Quickselect process for the random permutation $\sigma$. Clearly, $C_n = \sum_{1 \le i,j \le n} \xi(i,j)$. Moreover (see \cite[Proof of Proposition~3.2]{Hwa}), for each $2 \le k \le n$,
	$$
	\sum_{1 \le j <k}\xi(k,j)(\sigma) = \left\{\begin{array}{ll}k-1, & \text { if } \sigma \in A_{k} \\ 0, & \text { if } \sigma \notin A_{k},\end{array}\right.
\quad \text{while} \quad 
	\sum_{1 \le j<k} \xi(j, k)(\sigma)=\left\{\begin{array}{ll}0, & \text { if } \sigma \in A_{k} \\ 1, & \text { if } \sigma \notin A_{k}.\end{array}\right.
	$$
	Thus, we obtain
	$$
	C_n= \sum_{k=2}^n \left[(k-1) \mathds{1}_{A_k} + 1 - \mathds{1}_{A_k}\right] = n-1 + \sum_{k=2}^n (k-2) \mathds{1}_{A_k}.
	$$
	Noting that the events $A_k$, $2 \le k \le n$, are independent, this implies that $n^{-1} C_n - 1$ has the same distribution as
	$$
	n^{-1} \sum_{k=1}^n (k-2) B_k=n^{-1} \sum_{k=1}^n k B_{k+2} - n^{-1}((n-1)B_{n+1} + n B_{n+2}) - n^{-1} =: W_n'- R_n - n^{-1},
	$$
	where $B_k \sim {\rm Ber}(1/k)$ are independent for $1\le k \le n+2$.
	Thus, the triangle inequality implies
	\begin{equation}\label{eq:dfirst}
	\begin{split}
		d_K\left(n^{-1} C_n - 1 , D\right) & = d_K(W_n'- R_n - n^{-1}, D) = d_K(W_n'- R_n, D + n^{-1}) \\
		& \le d_K\left(W_n', D\right) + d_K\left(W_n' - R_n, W_n'\right) + d_K(D,D+n^{-1}) 
	\end{split}
	\end{equation}
	and similarly
	\begin{equation}\label{eq:dsecond}
	d_K\left(n^{-1} C_n - 1 , D\right) \geq d_K\left(W_n', D\right) - d_K\left(W_n' - R_n, W_n'\right) - d_K(D,D+n^{-1}).
	\end{equation}
	From Theorem~\ref{thm.ber} with $\theta=1$, $\beta=2$ and $l=1$, we derive
	\begin{equation}\label{eqn:W_n'}
	\frac{\underline{C} \log n}{n} \leq d_K\left(W_n', D\right) \leq \frac{\overline{C} \log n}{n}
	\end{equation}
	for $n\geq 2$ with absolute constants $\overline{C},\underline{C} \in (0,\infty)$.
		Since the density of $D$ is bounded by $e^{-\gamma}$ (see \eqref{eq:pub}), we obtain
	$$
	d_K(D,D+n^{-1}) \leq \frac{e^{-\gamma}}{n},
	$$
	while
	$$
	d_K\left(W_n' - R_n, W_n'\right) \leq \mathbf{P}\left\{R_n \neq 0\right\} =1-\Prob{B_{n+1}=B_{n+2}=0} =1- \left(1-\frac{1}{n+1}\right)\left(1-\frac{1}{n+2}\right)  \le \frac{2}{n}.
	$$
	Plugging these estimates and \eqref{eqn:W_n'} into \eqref{eq:dfirst} and \eqref{eq:dsecond}, we obtain
	$$
	\frac{\underline{C} \log n}{n} - \frac{e^{-\gamma} +2}{n} \leq d_K\left(n^{-1} C_n - 1 , D\right) \leq \frac{\overline{C} \log n}{n} + \frac{e^{-\gamma} +2}{n}
	$$
	for $n\geq 2$, which are the desired lower and upper bounds.
\end{proof}

\begin{proof}[Proof of Corollary~\ref{cor:IT}]
	By \cite[Section~5.1, Proposition~1]{KP07}, the re-scaled weighted depth $\widetilde W_{n}:=(W_{n} - n)/n$ has the same distribution as
	$$
\frac{1}{n} \sum_{k=1}^{n-1} k B_k' = \frac{1}{n} \sum_{k=1}^{n} k B_k' - B_n'=: W_n' - B_n',
	$$
	where $B_k' \sim {\rm Ber} \left(\frac{1+c_2/c_1}{k+c_2/c_1}\right)$, $1 \le k \le n$, are independent, and $c_1\in\mathbb{R}\setminus\{0\}$, $c_2\in\mathbb{R}$ are the parameters of the underlying tree evolution process with $1+c_2/c_1 >0$. Recall that $\beta=c_2/c_1$. By the triangle inequality we have
	\begin{equation}\label{eq:Ttr}
		d_K(\widetilde W_{n}, D_{1+\beta}) \le d_K\left(W_n', D_{1+\beta}\right) + d_K\left(W_n', W_n' - B_n'\right)
	\end{equation}
and
\begin{equation}\label{eq:Ttr2}
		d_K(\widetilde W_{n}, D_{1+\beta}) \ge d_K\left(W_n', D_{1+\beta}\right) - d_K\left(W_n', W_n' - B_n'\right).
\end{equation}
Theorem~\ref{thm.ber} implies
	\begin{equation}\label{eq:Tdk}
	\underline{C} T_n	\le d_K\left(W_n', D_{1+\beta}\right) \le \overline{C} T_n \quad \text{with} \quad T_n:= \begin{dcases} \frac{1 + \beta \log n}{n}, & \quad \beta \geq 0,\\  \frac{1}{n^{1+\beta}}, & \quad \beta \in (-1,0), \end{dcases}
	\end{equation}
	for $n\in\mathbb{N}$ with constants $\underline{C},\overline{C} \in (0,\infty)$ depending on $\beta$. Finally, note that
	\begin{equation}\label{eq:dsplit1}
		d_K\left(W_n', W_n' - B_n'\right) \le \Prob{B_n' > 0} = \frac{1+\beta}{n+\beta}.
	\end{equation}
Plugging \eqref{eq:Tdk} and \eqref{eq:dsplit1} into \eqref{eq:Ttr} and \eqref{eq:Ttr2} yields the desired upper and lower bounds.
\end{proof}

\begin{proof}[Proof of Theorem~\ref{thm:imp}]
We argue as in the proof of \cite[Theorem~1.1]{BG17} for proving the assertion in Theorem~\ref{thm:imp}(a) when $Y_k \equiv k$ for $k \in \{1,\hdots, n\}$. For $\overline W_n :=n^{-1}\sum_{k=\lceil \theta \rceil}^n kB_k$ with independent $B_k \sim {\rm Ber}(\theta/k)$ for $\theta>0$, note that the solution $f$ corresponding to $h \in \mathcal{H}_{1,1}$ from \cite[Theorem~4.9 and the proof of Lemma~4.8]{BG17} is twice differentiable on $(0,\infty)$ and
	\begin{equation*}
	d_{1,1}(\overline W_n, D_\theta) \le \sup_{f \in \mathcal{H}_{\theta, \theta/2}} \left|\E \left[(\overline W_n/\theta) f'(\overline W_n)-f'(\overline W_n+U)\right]\right|.
	\end{equation*} 
	Notice,
	\begin{align*}
	&\E \left[(\overline W_n/\theta) f'(\overline W_n)-f'(\overline W_n+U)\right]=\frac{1}{n}\sum_{k=\lceil \theta \rceil}^n \E\left[\frac{kB_k}{\theta} f'\left(\overline W_n^{(k)}+\frac{k}{n}B_k\right)\right] - \E \left[\int_0^1 f'(\overline W_n+u)du\right] \\
	& \, =\frac{1}{n} \sum_{k=\lceil \theta \rceil}^n \E \left[f'\left(\overline W_n^{(k)}+\frac{k}{n}\right) -f'\left(\overline W_n+\frac{k}{n}\right) \right]
	+\E \left[\frac{1}{n}\sum_{k=\lceil \theta \rceil}^n f'\left(\overline W_n+\frac{k}{n}\right)- \int_0^1 f'(\overline W_n+u)du\right]
	\end{align*} 
with $\overline W_n^{(k)}=\overline W_n-(k/n) B_k$ for $k\in\{\lceil \theta\rceil,\hdots,n\}$. Hence, using that $f' \in {\rm Lip}_{\theta/2}$ with $\|f'\|_{\R_+} \le \theta$ yields
	\begin{align*}
	d_{1,1}(\overline W_n, D_\theta) &\le \frac{\theta}{2n}\sum_{k=\lceil \theta \rceil}^n  \E \left[|\overline W_n^{(k)}-\overline W_n|\right] \\
	&\qquad + \E \sum_{k=1}^n \int_{\frac{k-1}{n}}^{\frac{k}{n}}\left| f' \left(\overline W_n+\frac{k}{n}\right)- f'(\overline W_n+u)\right| du + \frac{1}{n} \E \sum_{k=1}^{\lceil \theta \rceil -1}  \left|f'\left(\overline W_n+\frac{k}{n}\right)\right|\\
	& \le \frac{\theta}{2n}\sum_{k=\lceil \theta \rceil}^n \E \left[\frac{k}{n}B_k\right] + \frac{\theta}{2}\sum_{k=1}^n \int_{\frac{k-1}{n}}^{\frac{k}{n}} \left(\frac{k}{n}-u\right)du + \frac{\theta^2}{n}\\
	&\le \frac{\theta}{2} \left( \frac{1}{n^2}\sum_{k=1}^n k - \int_0^1 u \, du \right) +\frac{3\theta^2}{2n} = \frac{\theta}{2} \left( \frac{n(n+1)}{2 n^2} - \frac{1}{2} \right) +\frac{3\theta^2}{2n} =  \frac{\theta(6\theta +1)}{4n}.
	\end{align*} 
	Thus, the triangle inequality implies
	\begin{equation}\label{tri.eq}
	d_{1,1}(W_n,D_\theta) \le \frac{\theta(6\theta +1)}{4n} + d_{1,1}(W_n,\overline W_n).
	\end{equation}
	Now couple $(W_n)_{n \in \N}$ and $(\overline{W}_n)_{n \in \N}$ by using the same sequence $(B_k)_{\lceil \theta \rceil \le k \in \N}$ for both. Then one has that for $h \in \mathcal{H}_{1,1}$ as defined in \eqref{eqn:TestFunctionsH11},
	\begin{equation*}
	h(W_n)=h(\overline W_n) + h'(\overline W_n)(W_n-\overline W_n) + R_n
	\end{equation*}
	with 
	$$
	R_n := (W_n - \overline W_n) \int_0^1 \left[h'(\overline W_n + t(W_n - \overline W_n)) - h'(\overline W_n)\right] dt
	$$
	so that by the Lipschitz property of $h'$, we have $|R_n| \le \frac{1}{2}(W_n-\overline W_n)^2$. Now conditioning on $(B_k)_{\lceil \theta \rceil \le k \le n}$, we obtain
	\begin{align*}
	\E \left[h(W_n)-h(\overline W_n) |(B_k)_{\lceil \theta \rceil \le  k \le n}\right]&=h'(\overline W_n) \E [W_n-\overline W_n | (B_k)_{\lceil \theta \rceil \le  k \le n}] + \E [R_n | (B_k)_{\lceil \theta \rceil \le k \le n}]\\
	&=\E [R_n | (B_k)_{\lceil \theta \rceil \le k \le n}].
	\end{align*}
	Using the independence of $Y_k$ and $B_k$ for each $k$ and that $\E\, [Y_k]=k$, we have
	$$
	|\E [h(W_n)]-\E [h(\overline W_n)]| \le \E[|R_n|] \le \frac{1}{2}\E [(W_n-\overline W_n)^2] = \frac{1}{2n^2} \sum_{k=\lceil \theta \rceil}^n \frac{\theta \sigma_k^2}{k}. 
	$$
	Hence, from \eqref{tri.eq}, we obtain
	\begin{equation*}
	d_{1,1}(W_n,D_\theta) \le 
	\frac{\theta(6\theta +1)}{4n} + \frac{\theta}{2n^2} \sum_{k=\lceil \theta \rceil}^n\frac{\sigma_k^2}{k}
	\end{equation*}
	yielding the first assertion. 
	
	When $W_n=n^{-1}\sum_{k=1}^n kP_k$ with $P_k \sim {\rm Poi}(\theta/k)$, by \cite[Theorem~1.2]{BG17}, we have
	\begin{equation*}
		d_{1,1}(\overline W_n, D_\theta) \le \frac{\theta}{4n},
	\end{equation*} 
with $\overline W_n$ defined similarly as in the case of Bernoulli summands. Arguing similarly as above, first using the triangle inequality, then expanding $h(W_n)$ around $\overline W_n$, followed by a use of the independence of $(Y_k)_{k \in \N}$ and $(P_k)_{k \in \N}$, one obtains
\begin{equation*}
	d_{1,1}(W_n,D_\theta) \le \frac{\theta}{4n} + \frac{\theta}{2n^2} \sum_{k=1}^n\frac{\sigma_k^2}{k}\left(1+\frac{\theta}{k}\right) \le \frac{\theta}{4n} + \frac{\theta(1+\theta)}{2n^2} \sum_{k=1}^n\frac{\sigma_k^2}{k},
\end{equation*}
which completes the proof.
\end{proof}

\end{document}